\documentclass[10pt]{article}

\usepackage{graphicx}
\usepackage{subfigure}
\usepackage{algorithm}
\usepackage{algpseudocode}
\usepackage{amsmath}
\usepackage{amsthm}
\usepackage{amssymb}
\usepackage{caption}
\usepackage[]{todonotes} 
\usepackage{latexsym}
\usepackage{graphicx}
\usepackage{fullpage}
\usepackage{dsfont}
\usepackage{tikz}
\usetikzlibrary{calc, decorations.markings, intersections}
\usepackage{color}
\usepackage{chngcntr}
\usepackage{paralist}

\usetikzlibrary{decorations,arrows}
\usetikzlibrary{decorations.pathreplacing}
\usetikzlibrary{snakes}
\usetikzlibrary{arrows}
\usetikzlibrary{shapes}
\usetikzlibrary{backgrounds}

\newtheorem{theorem}{Theorem}[section]
\newtheorem{proposition}[theorem]{Proposition}
\newtheorem{lemma}[theorem]{Lemma}

\newcommand{\R}{\mathbb{R}}

\newcommand\1{{\mathbf 1}}

\newcommand{\binvar}[1]{\{ 0,1 \}^{#1}}
\newcommand{\conv}[1]{\convexh\left( #1\right)}

\DeclareMathOperator*{\convexh}{conv}
\DeclareMathOperator*{\argmax}{argmax}

\begin{document}
\title{Oracle-Based Algorithms for Binary Two-Stage Robust Optimization\thanks{This work has been supported by the German Research Foundation (DFG) under grant no. BU~2313/2 -- CL~318/14.}}

\author{Nicolas K\"ammerling \thanks{Institute of Transport Logistics, TU Dortmund University, Leonhard-Euler-Stra{\ss}e 2, 44227 Dortmund, Germany} \and Jannis Kurtz\thanks{Chair for Mathematics of Information Processing, RWTH Aachen University, Pontdriesch 12-14, 52062 Aachen, Germany}}

\date{} 

\providecommand{\keywords}[1]{\textit{#1}}


\maketitle

\begin{abstract}
In this work we study binary two-stage robust optimization problems with objective uncertainty. 
We present an algorithm to calculate efficiently lower bounds for the binary two-stage robust problem by solving alternately the underlying deterministic problem and an adversarial problem. For the deterministic problem any oracle can be used which returns an optimal solution for every possible scenario. We show that the latter lower bound can be implemented in a branch \& bound procedure, where the branching is performed only over the first-stage decision variables. All results even hold for non-linear objective functions which are concave in the uncertain parameters. As an alternative solution method we apply a column-and-constraint generation algorithm to the binary two-stage robust problem with objective uncertainty.

We test both algorithms on benchmark instances of the uncapacitated single-allocation hub-location problem and of the capital budgeting problem. Our results show that the branch \& bound procedure outperforms the column-and-constraint generation algorithm.
\end{abstract}

\section{Introduction.}\label{sec:intro}
The concept of robust optimization was created to tackle optimization problems with uncertain parameters. The basic idea behind this concept is to use uncertainty sets instead of probability distributions to model uncertainty. More precisely it is assumed that all realizations of the uncertain parameters, called \textit{scenarios}, are contained in a known uncertainty set. Instead of optimizing the expected objective value or a given risk-measure as common in the field of stochastic optimization, in the robust optimization framework we calculate solutions which are optimal in the worst case and which are feasible for all scenarios in the uncertainty set. 

The concept was first introduced in \cite{soyster1973}. Later it was studied for combinatorial optimization problems with discrete uncertainty sets in \cite{kouvelis}, for conic and ellipsoidal uncertainty in \cite{roconv,bentalRobust}, for semi-definite and least-square problems in \cite{el1998robust,el1997robust} and for budgeted uncertainty in \cite{bertsimas04priceofrobustness,bertsimas2003combinatorial}. An overview of the robust optimization literature can be found in \cite{buchheim2018robust,bertsimas2011theory,aissi_minmax_survey,ben2009robust}. 

The so called robust counterpart is known to be NP-hard for most of the classical combinatorial problems, although most of them can be solved in polynomial time in its deterministic version; see \cite{kouvelis}. Furthermore it is a well-known drawback of this approach that the optimal solutions are often too conservative for practical issues \cite{bertsimas04priceofrobustness}. To obtain better and less-conservative solutions several new ideas have been developed to improve the concept of robustness; see e.g. \cite{kouvelis,fischetti_lightrobust,schoebel_generalizedrobust,liebchen_recovrobust,adjiashvili2015bulk}. 

Inspired by the concept of two-stage stochastic programming a further extension of the classical robust approach which attained increasing attention in the last decade is the concept of \textit{two-stage robustness}, or sometimes called \textit{adjustable robustness}, first introduced in \cite{ben2004adjustable}. The idea behind this approach is tailored for problems which have two different kinds of decision variables, first-stage decisions which have to be made \textit{here-and-now} and second-stage decisions which can be determined after the uncertain parameters are known, sometimes called \textit{wait-and-see} decisions. As in the classical robust framework it is assumed that all uncertain scenarios are contained in a known uncertainty set and the worst-case objective value is optimized. The main difference to the classical approach is that the second-stage decisions do not have to be made in advance but can be chosen as the best reaction to a scenario after it occured. This approach can be modeled by min-max-min problems in general. Famous applications occur in the field of network design problems where in the first stage a capacity on an edge must be bought such that, after the real costs on each edge are known, a minimum cost flow is sent from a source to a sink which can only use the bought capacities~\cite{bertsimas2013adjustableapprox}. An overview of recent results for two-stage robustness can be found in \cite{yanikouglu2017adjustable}. Several concepts closely related to the two-stage robust concept were introduced in \cite{liebchen_recovrobust,adjiashvili2015bulk,buchheimkurtzconvex}.

In this work we study binary two-stage robust optimization problems. We consider underlying deterministic problems of the form
\begin{equation}\label{eq:deterministicprob}\tag{CP}
\min_{(x,y)\in Z} f(x,y,c)
\end{equation}
where $f:Z\times \R^m \to \R$, the set $Z\subseteq\binvar{n_1+n_2}$ contains all incidence vectors of the feasible solutions and is assumed to be non-empty, $c\in\R^{m}$ is a given parameter vector and $f(x,y,\cdot )$ is concave for each given $(x,y)\in Z$. The variables $x$ are called \textit{first-stage solutions} and the variables $y$ are called \textit{second-stage solutions}. We assume that the vector $c$ is uncertain and all possible realizations $c$ are contained in a convex uncertainty set $U\subset \R^{m}$. The binary two-stage robust problem is then defined by
\begin{equation}\label{eq:two-stageprob}\tag{2RP}
\min_{x\in X}\max_{c\in U}\min_{y\in Y(x)} f(x,y,c)
\end{equation}
where $X\subset \binvar{n_1}$ is the projection of $Z$ onto the $x$-variables, i.e. \[X:=\left\{ x\in \binvar{n_1} \ | \ \exists \ y\in \binvar{n_2} : (x,y)\in Z\right\},\] and $Y(x):=\left\{ y\in \binvar{n_2} \ | \ (x,y)\in Z\right\}$. Note that all results presented in this paper are still valid, if the recourse variables are non-integer. We do not consider uncertainty affecting the constraints of the problem which is a situation often occuring in practice for most of the classical combinatorial optimization problems. Problem \eqref{eq:two-stageprob} can be interpretated as follows: In the first stage, before knowing the precise uncertain vector $c$, the decisions $x\in X$ have to be made. Afterwards, when the cost-vectors are known, we can choose the best feasible second-stage solution $y\in Y(x)$ for the given costs. As usual in robust optimization we measure the worst-case over all possible scenarios in~$U$.
Note that by our definition of the set $Y(x)$ and since the uncertainty only affects the objective function, there always exists a feasible second-stage solution $y\in Y(x)$ for each first-stage solution $x\in X$.

Problem \eqref{eq:two-stageprob} has been already studied in the literature and several exact algorithms as well as approximation algorithms have been proposed; see Section \ref{sec:literature}. While several of the existing methods are able to handle uncertainty in the constraints it is often assumed that a polyhedral description of the sets $X$ and $Y(x)$ is given. Besides the latter limitation most of the methods are based on dualizations or reformulations which destroy the structure of the original problem \eqref{eq:deterministicprob}. Often the uncertainty set is even restricted to be a polyhedron. In this work we derive the first oracle-based exact algorithm which solves Problem \eqref{eq:two-stageprob} for any deterministic problem by solving alternately the deterministic Problem \eqref{eq:deterministicprob} and an adversarial problem presented later. For the deterministic problem any oracle can be used which returns an optimal solution of \eqref{eq:deterministicprob} for every possible scenario in $U$. The advantage of the latter method is that the structure of the underlying problem is preserved and any preliminary algorithms which were derived for the underlying problem can be used. Furthermore our algorithm works for most of the common convex uncertainty sets. Additionally we apply the column-and-constraint generation algorithm (CCG) presented in \cite{zeng2013twostage} to Problem \eqref{eq:two-stageprob} and compare it to our new method.

In Section \ref{sec:literature} we will give an overview of the literature related to two-stage robust optimization problems. In Section \ref{sec:lineartwostage} we derive an oracle-based branch \& bound procedure to solve Problem \eqref{eq:two-stageprob}. Furthermore we apply the results in \cite{zeng2013twostage} to Problem \eqref{eq:two-stageprob}. Finally in Section \ref{sec:hub} we apply both methods to the uncapacitated singe-allocation hub-location problem and the capital budgeting problem and test it on classical benchmark instances from the literature.

Our main contributions:
\begin{itemize}
\item  We adapt the oracle-based algorithm derived in \cite{buchheim2016min} and show that it can be used to calculate a lower bound for Problem \eqref{eq:two-stageprob} which can be implemented in a branch \& bound procedure where the branching is performed over the first-stage solutions. The calculation of the lower bound can be applied to the common convex uncertainty sets and is done by alternately calling an adversarial problem over $U$ and an oracle which returns an optimal solution of Problem \eqref{eq:deterministicprob} for a given scenario $c\in U$. Therefore any solution algorithm of the deterministic problem can be used to calculate this lower bound.
\item We apply the CCG algorithm presented in \cite{zeng2013twostage} to Problem \eqref{eq:two-stageprob} and show that calculating the upper bound can also be done by the same oracle-based algorithm as above.
\item We apply the branch \& bound procedure and the CCG algorithm to the uncapacitated single-allocation hub-location problem and the capital budgeting problem and show that the branch \& bound procedure outperforms the CCG algorithm.
\end{itemize}

\subsection{Related Literature}\label{sec:literature}
Linear two-stage robust optimization or sometimes called adjustable robust optimization was first introduced in~\cite{ben2004adjustable}. The authors show that the problem is NP-hard even if $X$ and $Y$ are given by linear uncertain constraints and all variables are real; see also \cite{minoux2011twostage}. In \cite{ben2004adjustable} the authors propose to approximate the problem by assuming that the optimal values of the wait and see variables $y$ are affine functions of the uncertain parameters. These so called \textit{affine decision rules} were studied in the robust context in several articles for the case of real recourse; see e.g. \cite{atamturk2007two,ben2005retailer,calafiore2008multi,chen2009uncertain,iancu2010adaptive,kuhn2011primal,shapiro2011dynamic,vayanos2012constraint}.
Furthermore in several works special cases are derived for which a decision rule structure is known which is optimal; see \cite{bertsimas2010optimality,iancu2013supermodularity,bertsimas2012power}. Further non-linear decision rules are studied in \cite{yanikouglu2017adjustable}.

Lower bounds for two-stage robust problems can be derived by considering a finite subset of scenarios in~$U$. Then for each selected scenario $c$ a duplication of the second-stage solution $y^c$ is added to the problem, see \cite{hadjiyiannis2011scenario,campi2004decision,ayoub16}. The authors in \cite{bertsimas2016duality} first dualize the inner minimization and maximization problem and then apply the latter finite scenario approach to the dual problem to obtain stronger lower bounds. Note that while the finite scenario approach can also be applied to the case when the second-stage solutions are integers, for the dualization approach the second-stage variables have to be relaxed to real variables. Unfortunately both lower bounds can not be used in a branch \& bound scheme since for a complete fixation of the first-stage variables the bounds are not necessarily exact. 

Exact methods for real recourse are based on the idea of Benders' decomposition, see \cite{thiele2009robust,bertsimas2013adaptive,jiang2012benders,gabrel2014robust} or column-and-constraint generation \cite{zeng2013twostage,bertsimas2018scalable}. Note that for the Benders' decomposition approaches the second-stage solutions have to be real since dualizations of the second-stage problem are used. In contrast to this the CCG algorithm even works for integer recourse, see \cite{zhao2012exact}. We will apply the latter method to our problem in Section \ref{sec:CCG_linear}.

For the case of integer recourse, i.e. the second-stage variables $y$ are modeled as integer variables, decision rules have been applied to Problem \eqref{eq:two-stageprob} in \cite{bertsimas2015design,bertsimas2014binarydecisionrules} to approximate the problem. Another approximation approach is called $k$-adaptability and was introduced in \cite{bertsimas2010finite}. The idea is to calculate $k$ second-stage solutions in the first-stage and allow to choose the best out of these solutions in the second-stage. Clearly since the set of possible second-stage solutions is restricted compared to the original problem, this idea leads to an approximation of the problem. Solution methods and the quality of this approximation were studied in \cite{bertsimas2010optimality,wiesemann_twostage,subramanyam17}. In \cite{wiesemann_twostage} it is shown that the $k$-adaptability problem is exact if $k$ is chosen larger than the dimension of the problem. The authors in \cite{buchheimkurtzconvex,buchheim2018complexity,eufingerrobust} apply the $k$-adaptability concept to one-stage combinatorial problems to calculate a set of solutions which is worst-case optimal if for each scenario the best of these solutions can be chosen. They furthermore show that solving this problem can be done in polynomial time if an oracle for the deterministic problem exists and if the number of calculated solutions is larger or equal to the dimension of the problem. To solve the problem in the latter case they present an oracle-based algorithm which we will use in Section \ref{sec:lineartwostage}. The $k$-adaptability concept was also applied to the case that the uncertain parameters follow a discrete probability distibution \cite{buchheim2019k}. 

Besides the exact algorithm in \cite{zhao2012exact,zeng2013twostage} approximation methods based on uncertainty set splitting were derived in the literature to approximate two-stage robust problems with integer recourse; see \cite{postek16,bertsimas16_2}.

For two-stage robust problems with non-linear robust constraints decision rules have been applied in \cite{takeda2008adjustable,nagy2003robust}. The two-stage problem is studied for second order conic optimization problems in \cite{boni2008adjustable}. In \cite{ben2015deriving,marandi2017extending} the authors derive robust counterparts of uncertain non-linear constraints. Note that all the latter results were developed for real second-stage solutions.

While this work was under peer review a similar approach to solve two-stage robust optimization problems with uncertainty only affecting the objective function was published; see \cite{arslan2019decomposition}. The authors study Problem \eqref{eq:two-stageprob} with linear objective functions and mixed-integer recourse variables, while the set $Y(x)$ is modeled by linear constraints. They study a relaxation of the lower bound presented in Section \ref{sec:lineartwostage} which is implemented in a branch \& bound procedure. In contrast to the algorithm described in this work, the method in \cite{arslan2019decomposition} is not based on the use of oracles for the deterministic problem. Therefore it can not make use of fast solution methods for \eqref{eq:deterministicprob} as combinatorial algorithms or compact formulations with uncertain parameters appearing in the constraints; see Section \ref{sec:hub}.
 
\section{Binary Two-Stage Robustness}\label{sec:lineartwostage}
In this section we analyze the binary two-stage robust problem \eqref{eq:two-stageprob} with convex uncertainty sets $U$ and derive general lower bounds which can be calculated by an oracle-based algorithm and which can be implemented in a branch \& bound procedure. The branching will be done over the first-stage solutions.

The classical approach to derive lower bounds in a branch \& bound procedure is relaxing the integrality and solving the relaxed problem. Applying this approach to the second-stage decisions of problem \eqref{eq:two-stageprob} is not useful, since for a given $x\in X$ and $c\in U$ an optimal solution of the relaxed second-stage problem may not be contained in $\conv{Y(x)}$, e.g. if the relaxation of $Y(x)$ is a polytope which is not integral. It may be even the case that a linear description of $\conv{Y(x)}$ is not known. Therefore, even if all first-stage variables are fixed, the lower bound obtained by relaxing the second-stage solution variables would not necessarily be exact and an optimal solution can not be guaranteed using a branch \& bound scheme. In the following lemma we derive a lower bound for Problem \eqref{eq:two-stageprob} which is exact if all first-stage solutions are fixed.
\begin{lemma}\label{lem:lowerboundlinear}
Given $U\subset \R^{m}$, then 
\begin{equation}\label{eq:convlowerbound}\tag{LB}
 \max_{c\in U} \min_{(x,y)\in\conv{Z}} \ f(x,y,c)
\end{equation}
is a lower bound for Problem \eqref{eq:two-stageprob}.
\end{lemma}
\begin{proof}
By changing the order of the outer minimum and the inner maximum in Problem \eqref{eq:two-stageprob} we obtain the inequality
\[
\min_{x\in X}\max_{c\in U}\min_{y\in Y(x)} f(x,y,c)\ge \max_{c\in U}\min_{x\in X}\min_{y\in Y(x)} f(x,y,c).
\]
Merging the two minimum expressions and using $Z\subseteq \conv{Z}$ yields 
\[
\max_{c\in U}\min_{x\in X}\min_{y\in Y(x)} f(x,y,c)\ge \max_{c\in U} \min_{(x,y)\in\conv{Z}} \ f(x,y,c),
\]
which proves the result.
\end{proof}
Note that, since $f$ is concave in $c$ and since the pointwise minimum of concave functions is always concave, we have to maximize a concave objective function in Problem \eqref{eq:convlowerbound}. In \cite{buchheimkurtzconvex} the authors analyze Problem \eqref{eq:convlowerbound} for the case that $f$ is a linear function in $(x,y)$ and $c$. They prove that it can be solved in oracle-polynomial time, i.e. by a polynomial time algorithm if solving the deterministic problem \eqref{eq:deterministicprob} is done by an oracle in constant time. Furthermore if we fix a solution $x\in X$, then the bound \eqref{eq:convlowerbound} is exact, which we prove in the following.
\begin{proposition}\label{prop:exactIfFixed}
If all first-stage variables are fixed then \eqref{eq:convlowerbound} is equal to the exact objective value of the fixed first-stage solution.
\end{proposition}
\begin{proof}
Let $\bar x\in X$ be the fixed first-stage solution, then it holds
\[
\left\{ (x,y)\in\conv{Z} \ | \ x=\bar x \right\} = \left\{\bar x\right\} \times \conv{Y(\bar x )} .
\]
Clearly problem
\[
\max_{c\in U} \min_{(x,y)\in\left\{ \bar x\right\} \times \conv{Y(\bar x )}} \ f(x,y,c)
\]
is equivalent to
\begin{equation}\label{eq:evaluateobjfct}
\max_{c\in U} \min_{y\in Y(\bar x)} f(\bar x,y,c)
\end{equation}
which proves the result.
\end{proof}
The result of Proposition \ref{prop:exactIfFixed} indicates that the lower bound \eqref{eq:convlowerbound} can be integrated in a branch \& bound procedure.
 
In \cite{buchheimkurtzconvex} it was proved that, given an oracle to solve the deterministic problem over $Y(\bar x)$ for each given $\bar x$, if $f$ is linear in $(x,y)$ and $c$ and under further mild assumptions, Problem \eqref{eq:evaluateobjfct} can be solved in oracle-polynomial time. Together with Proposition \ref{prop:exactIfFixed} a direct consequence is that, if the dimension $n_1$ of the first-stage solutions is fixed, then we can enumerate over all possible first-stage solutions and compare the objective values in oracle-polynomial time. Hence, Problem \eqref{eq:two-stageprob} can be solved in polynomial time given an oracle for the optimization problem over $Y(x)$ for each $x\in X$.

The authors in \cite{buchheimkurtzconvex} present a practical algorithm, based on the idea of column-generation for the case that $f$ is a linear function. Applied to the more general Problem \eqref{eq:convlowerbound} the algorithm can be derived as follows: The algorithm starts with a subset of solutions $Z'\subset Z$, leading to problem
\begin{equation}\label{eq:convlowerboundsubset}
\max_{c\in U} \min_{z\in\conv{Z'}} f(z,c) ,
\end{equation}
and then iteratively adds new solutions to $Z'$ until optimality can be ensured. The solution which is added in each iteration is the one which has the largest impact on the optimal value. To find this solution Problem \eqref{eq:convlowerboundsubset} can be reformulated by applying a level set transformation. The reformulation is given by 
\begin{equation}\label{eq:dualprobalgorithm}
\begin{aligned}
\max & \quad \mu \\
s.t. & \quad f(z,c)\ge \mu \ \ \forall z\in Z' \\
& \quad \mu\in\R, \ c\in U.
\end{aligned}
\end{equation}
For an optimal solution $(\mu^*,c^*)$ of the latter problem, we search for the solution $z\in Z$ which most violates the constraint $f(z,c^*)\ge \mu^*$, i.e. the solution with the largest improvement on the optimal value of Problem~\eqref{eq:convlowerboundsubset}. The latter task can be done by minimizing the objective function $f(z,c^*)$ over all $z\in Z$, i.e. solving the deterministic problem \eqref{eq:deterministicprob} under scenario $c^*$ by using any exact algorithm. If we can find a $z^*\in Z$ such that $f(z^*,c^*) < \mu^*$, then we add $z^*$ to $Z'$ and repeat the procedure. If no such solution can be found, then $f(z,c^*)\ge \mu^*$ holds for all $z\in Z$ and therefore $\mu^*$ is the optimal value of \eqref{eq:convlowerbound}. 
The procedure described above is presented in Algorithm \ref{alg:columngenerationlinear}.

\begin{algorithm}[htb]
\caption{~~Algorithm to calculate the lower bound~\eqref{eq:convlowerbound}}
\label{alg:columngenerationlinear}                           
\begin{algorithmic}[1]
  \Require Convex $U\subset\R^{m}$, $Z\subseteq\binvar{n_1+n_2}$
  \Ensure Optimal value of Problem~\eqref{eq:convlowerbound} and a set of feasible solutions $Z'\subseteq Z$
  \State Choose any $z_0\in Z$ and set $Z':=\left\{ z_0\right\}$
  \Repeat
  \State Calculate an optimal solution $(\mu^*, c^*)$ of\vspace*{-2ex}
  $$\max\;\{\mu\mid f(z,c)\ge \mu \ \ \forall z\in Z',~\mu\in\R,~c\in U\}$$\vspace*{-4ex}\label{step:dual}
  \State Calculate an optimal solution $z^*$ of\vspace*{-2ex}
  $$\min_{z\in Z}f(z,c^*)\;$$\vspace*{-4ex}\label{step:if}
  \State Add $z^*$ to $Z'$
  \Until{$f(z^*,c^*) \ge \mu^*$}\\
\Return $\mu^*, Z'$
\end{algorithmic}
\end{algorithm}

Note that the Problem in Step \ref{step:dual} depends on the uncertainty set~$U$ and on the properties of $f$. If $f$ is a linear function, for polyhedral or ellipsoidal uncertainty sets this is a continuous linear or quadratic problem, respectively. Both problems can be solved by the latest versions of optimization software like CPLEX~\cite{cplex128}. Therefore the algorithm can be implemented for each deterministic problem by using any exact algorithm to solve the deterministic problem in Step \ref{step:if}. The main advantage of this feature is that we do not have to restrict to deterministic problems which can be modeled by a linear compact formulation as it is the case in \cite{arslan2019decomposition}. Instead we can use any combinatorial algorithm or even mixed-integer formulations where the uncertain parameters appear in the constraints; see Section \ref{sec:hub}. We only require an arbitrary procedure which returns an optimal solution for the given scenario. In~\cite{eufingerrobust} the authors applied the latter algorithm to the min-max-min robust capacitated vehicle routing problem and showed that on classical benchmark instances the number of iterations of Algorithm \ref{alg:columngenerationlinear} is significantly smaller than the dimension of $Z$ in general. 

Note that besides the optimal value of Problem \eqref{eq:two-stageprob} the algorithm returns a set of feasible solutions $Z'\subseteq Z$ and not a solution in $\conv{Z}$. By the correctness of the algorithm the optimal solution in $\conv{Z}$ must be contained in $\conv{Z'}$ and could be calculated by finding the optimal convex combination of the solutions in $Z'$ which can be done by solving the problem
\begin{equation}\label{eq:optimalConvComb}
\max_{c\in U} \min_{\substack{\lambda\ge 0\\ \sum_{z\in Z'}\lambda_z = 1\\ \tilde z=\sum_{z\in Z'}\lambda_z z}} f(\tilde z,c)
\end{equation}
for the given set $Z'$. If $f$ is linear in $z$ and $c$ then the latter problem is equivalent to
\[
\min_{\substack{\lambda\ge 0\\ \tilde z=\sum_{z\in Z'}\lambda_z z}} \max_{c\in U} f(\tilde z,c)
\]
and by dualizing the inner maximization problem over $U$ this is a continuous linear or quadratic problem for polyhedral or ellipsoidal uncertainty, respectively.
Nevertheless in our branch \& bound procedure for non-linear functions $f$ the set $Z'$ is sufficient as we will see in Section \ref{sec:branchandboundlinear}. A practical advantage of the set $Z'$ is that it contains a set of second-stage policies which can be used in practical applications. Instead of solving the second-stage problem each time after a scenario occured, which may be a computationally hard problem, we can choose the best of the pre-calculated second-stage policies in $Z'$ for the actual scenario. The latter task can be done by just comparing the objective values of all solutions in $Z'$ for the given scenario. Note that the returned set of solutions need not contain the optimal solution for each scenario. Nevertheless we will show in Section \ref{sec:computations} that the calculated solutions perform very well in average over random scenarios in $U$.


\subsection{Oracle-Based Branch \& Bound Algorithm}\label{sec:branchandboundlinear}
Using the results of the previous section we can easily derive a classical branch \& bound procedure to solve Problem \eqref{eq:two-stageprob}. The idea is to branch over the first-stage solutions $x\in X$ and to calculate the lower bound \eqref{eq:convlowerbound} in each node of the branch \& bound tree to possibly prune the actual branch of nodes. All necessary details needed to implement a branch \& bound procedure are presented in the following.
\paragraph{Handling Fixations}
In each node of the branch \& bound tree we have a given set of fixations for the $x$-variables, i.e. a set of  indices $I_0\subset [n_1]$ such that $x_i=0$ for each $i\in I_0$ and a given set of indices $I_1\subset [n_1]\setminus I_0$ such that $x_i=1$ for each $i\in I_1$. All indices in $[n_1]\setminus \left( I_0\cup I_1\right)$ are free. Therefore in each node for the given fixations we have to solve the problem
\begin{equation}\label{eq:convlowerboundfixations}
\max_{c\in U} \min_{\substack{(x,y)\in\conv{Z} \\ x_i=0 \ \forall \ i\in I_0 \\ x_i=1 \ \forall \ i\in I_1}}  f(x,y,c) 
\end{equation}
or to decide if the latter problem is infeasible. It is easy to see that the latter problem, if it is feasible, can be solved by Algorithm \ref{alg:columngenerationlinear} by including the given fixations into the set $Z$. Note that here the oracle for the deterministic problem must be able to handle variable-fixations. Nevertheless for most of the classical problems fixations can easily be implemented in most algorithms.
\paragraph{Warm Starts}
In each node of the branch \& bound tree Algorithm \ref{alg:columngenerationlinear} returns a set $Z'\subset Z$ of feasible solutions satisfying the given fixations. For each possible child-node we can select the set $Z''\subset Z'$ of solutions which satisfy the new fixations and warm-start Algorithm \ref{alg:columngenerationlinear} with the set $Z''$ in the child node.
\paragraph{Branching Strategy}
An easy branching strategy can be established as follows: For the calculated set of solutions $Z'$ returned by Algorithm \ref{alg:columngenerationlinear} we define the vector $\bar x\in [0,1]^{n_1}$ by
\begin{equation}\label{eq:dualsolutionavg}
\bar x_i =  \frac{1}{|Z'|}\sum_{(x,y)\in Z'} x_i 
\end{equation}
for all $i\in [n_1]$, i.e. the value $\bar x_i$ is the fraction of solutions in $Z'$ for which $x_i=1$ holds. We can then use any of the classical branching rules, e.g. we can decide to branch on the index $i$ for which the value $\bar x_i$ is closest to $0.5$.

Another computationally more expensive approach is to calculate the optimal convex combination of the solutions in $Z'$, i.e. after calculating the optimal $Z'$ by Algorithm \ref{alg:columngenerationlinear} we calculate an optimal solution $\lambda^*$ of Problem \eqref{eq:optimalConvComb} and define \begin{equation}\label{eq:dualsolutionopt}\bar x = \sum_{z=(x,y)\in Z'} \lambda_z^* x.\end{equation} Now we can again use any classical branching-strategy on $\bar x$. Note that if a first-stage variable has the same value in each of the solutions in $Z'$ then also the corresponding entry of $\bar x$ has this value.

When going over to the next open branch \& bound node to be processed, we choose the one with the smallest lower-bound.
\paragraph{Calculating Feasible Solutions}
In each node of the branch \& bound tree we want to find a feasible solution to update the upper bound on our optimal value. We do this as follows: In each branch \& bound node Algorithm \ref{alg:columngenerationlinear} calculates a set $Z'\subseteq Z$ of feasible solutions. If all of the generated solutions in $Z'$ have the same first-stage solution $x$, then the optimal solution of \eqref{eq:convlowerboundfixations} has binary first-stage variables and we obtain a feasible solution $x\in X$ which has the objective value $\mu^*$ returned by the algorithm. If the first-stage variables are not the same for all $z\in Z'$ then we can either choose an arbitrary first-stage solution given by any $z\in Z'$ or we can calculate the objective value of all first-stage solutions in $Z'$ and choose the one with the best objective value. To this end we have to solve
\[
\max_{c\in U}\min_{y\in Y(\tilde x)} f(\tilde x,y,c),
\]
for any first-stage solution $\tilde x$ given in $Z'$. Note that the latter problem again can be solved by Algorithm \ref{alg:columngenerationlinear} replacing the deterministic problem in Step \ref{step:if} by
\[
\min_{y\in Y(\tilde x)}f(\tilde x,y,c^*) .
\]
If $X=\binvar{n_1}$, as it is the case for the hub-location problem (see Section \ref{sec:computations}), then calculating all objective values as above can be avoided and finding a good feasible solution can be done by rounding each component of the vector $\bar x$ calculated in the latter paragraph.

\subsection{Oracle-Based Column-and-Constraint Algorithm}\label{sec:CCG_linear}
In \cite{zeng2013twostage} a column-and-constraint generation method (CCG) was introduced to solve two-stage robust problems with real recourse variables. In \cite{zhao2012exact} the authors show how the algorithm can be applied to two-stage robust problems with mixed-integer recourse variables. In both cases the algorithm is studied for problems with uncertain constraints. In this section we will apply the algorithm to  Problem \eqref{eq:two-stageprob}, i.e. to the special case of objective uncertainty, and show that we can again use Algorithm \ref{alg:columngenerationlinear} to solve one crucial step in the CCG. In the following we derive the CCG algorithm for Problem \eqref{eq:two-stageprob}. For more details see~\cite{zeng2013twostage,zhao2012exact}.

Using a level set transformation Problem \eqref{eq:two-stageprob} can be reformulated by 
\begin{align*}
\min & \ \ \mu \\
s.t. \quad & \mu\ge \max_{c\in U}\min_{y\in Y(x)} f(x,y,c) \\
& x\in X, \ \mu\in\R.
\end{align*}
If we choose any finite subset of scenarios $\left\{c^1,\ldots ,c^l\right\}\in U$ we obtain the lower bound
\begin{align*}
\min & \ \ \mu \\
s.t. \quad & \mu\ge \min_{y\in Y(x)} f(x,y,c^i) \ \ i=1,\ldots ,l \\
& x\in X, \ \mu\in\R,
\end{align*}
which is equivalent to problem
\begin{equation}\label{eq:lowerboundCCG}
\begin{aligned}
\min & \ \ \mu \\
s.t. \quad & \mu\ge f(x,y^i,c^i) \ \ i=1,\ldots ,l \\
& x\in X, \ \mu\in\R, \ y^i\in Y(x) \ \ i=1,\ldots ,l .
\end{aligned}
\end{equation}
The algorithm in \cite{zeng2013twostage} now iteratively calculates an optimal solution $(x^*,\mu^*)$ of the latter problem \eqref{eq:lowerboundCCG}, which is a lower bound for Problem \eqref{eq:two-stageprob}, and afterwards calculates a worst-case scenario $c^{l+1}\in U$ by
\begin{equation}\label{eq:worstcasescenarioCCG}
c^{l+1} = \argmax_{c\in U}\min_{y\in Y(x^*)} f(x^*,y,c).
\end{equation}
The optimal value of Problem \eqref{eq:worstcasescenarioCCG} is the objective value of solution $x^*\in X$ and therefore an upper bound for Problem \eqref{eq:two-stageprob}. Afterwards new variables $y^{l+1}$ and the constraint 
\[
\mu\ge f(x,y^{l+1},c^{l+1})
\]
are added to Problem \eqref{eq:lowerboundCCG} and we iterate the latter procedure until
\[
\mu^*\ge \max_{c\in U}\min_{y\in Y(x^*)} f(x^*,y,c) .
\]
Clearly a solution $(x^*,\mu^*)$ fulfilling the latter condition is optimal for Problem \eqref{eq:two-stageprob}. Following the proof of Proposition \ref{prop:exactIfFixed} the worst-case scenario in \eqref{eq:worstcasescenarioCCG} can be calculated by Algorithm \ref{alg:columngenerationlinear}. This can be done since we do not consider uncertainty in the constraints, while in the more general framework in \cite{zeng2013twostage} this is not possible.

The main difference of the latter procedure to our branch \& bound algorithm is that in a branch \& bound node only a subset of first-stage variables is fixed while the rest is relaxed. Then we use Algorithm \ref{alg:columngenerationlinear} to calculate a lower bound for the given fixations. In the CCG procedure in each iteration a first-stage solution is calculated by Problem \eqref{eq:lowerboundCCG} and therefore all variables are fixed when Algorithm \ref{alg:columngenerationlinear} is applied to calculate the worst-case scenario. Nevertheless the number of constraints and the number of variables of Problem \eqref{eq:lowerboundCCG} increase iteratively, since each second-stage variable has to be duplicated in each iteration, while in the branch \& bound procedure we always iterate over the same number of first-stage variables. In Section \ref{sec:computations} we will compare both algorithms on benchmark instances of the uncapacitated single-allocation hub location problem and the capital budgeting problem.

\section{Applications}
\subsection{The Uncapacitated Single-Allocation Hub Location Problem with Uncertain Demands}
\label{sec:hub} 
In this section the oracle-based branch \& bound algorithm is exemplarily applied to the single-allocation hub location problem which can be naturally defined as a two-stage problem. Furthermore due to its quadratic objective function it perfectly fits into the non-linear framework. 

Hub-location problems address the strategic planning of a transportation network with many sources and sinks. In many applications sending all commodities over direct connections would be too expensive in operation. Instead, some locations are considered to serve as transshipment points and are then called hubs. Thus, strongly consolidated transportation links are established. The bundling of shipments usually outweighs the additional costs of hubs and detours. Important applications of this problem arise in air freight \cite{jaillet1996airline}, postal and parcel transport services \cite{ernst1996efficient}, telecommunication networks \cite{klincewicz1998hub} and public transport networks \cite{nickel2001hub}. The recent surveys of \cite{alumur2008network} and \cite{campbell2012twenty} provide a comprehensive overview of the various variations and solution approaches of the hub location problem.

The main source of uncertainty in single-allocation hub location problems are demand fluctuations. Thus, it is important to include this uncertainty when deciding hub locations and allocations of the nodes to the hubs. Installing a hub is a long-term decision which lasts for many years or even for several decades. Nonetheless, the allocation to the hub nodes are mid-to-short-term decisions as they can be changed over time. In \cite{rostami2018stochastic} the variable allocation variant for single-allocation hub location problems under stochastic demand uncertainty is proposed. 

We consider a directed graph $G=(N,A)$, where $N = \{1,2, \ldots, n\}$ corresponds to the set of nodes that denote the origins, destinations, and possible hub locations, and $A$ is a set of arcs that indicate possible direct links between the different nodes. Let~$w_{ij}\ge 0$ be the amount of flow to be transported from node~$i$ to node~$j$ and $d_{ij}$ the distance between two nodes $i$ and $j$. We denote by~$O_i=\sum_{j\in N}w_{ij}$ and $D_i=\sum_{j\in N}w_{ji}$ the total outgoing flow from node~$i$ and the total incoming flow to node~$i$, respectively. For each~$k\in N$, the value $f_{k}$ represents the fixed set-up cost for locating a hub at node~$k$. The cost per unit of flow for each path~$i-k-m-j$ from an origin node~$i$ to a destination node~$j$ passing through hubs~$k$ and~$m$ respectively, is~$\chi d_{ik}+\alpha d_{km}+\delta d_{m j}$, where~$\chi$,~$\alpha$, 
and~$\delta$ are the nonnegative collection, transfer, and distribution costs respectively and $d_{ik}$, $d_{km}$, and $d_{m j}$ are the distances between the given pairs of nodes. Typically $\alpha\le \min\left\{ \chi, \delta\right\}$ since otherwise using a hub would not be beneficial. Note that if hub nodes are fully interconnected, every path between an origin and a destination node will contain at least one and at most two hubs.
The SAHLP consists of selecting a subset of nodes as hubs and assigning the remaining nodes to these hubs such that each spoke node is assigned to exactly one hub with the objective of minimizing the overall costs of the network. 

To formulate the SAHLP, we follow the first formulation of this problem introduced by O'Kelly \cite{o1987quadratic}. Two types of decision variables are introduced. First, the
\[
x_{k}=\begin{cases}
\begin{array}{ll}
  1 & \textnormal{ if node~$k$ is a hub node}\\
  0 & \textnormal{ otherwise.}
\end{array}
\end{cases}
\]
variables indicate whether a node is used as hub in the transportation network. Second, the 
\[
y_{ik}=\begin{cases}
\begin{array}{ll}
  1 & \textnormal{ if node~$i$ is allocated to a hub located at node $k$}\\
  0 & \textnormal{ otherwise.}
\end{array}
\end{cases}
\]
variables show how the nodes are allocated to the hub nodes.
SAHLP can then be formulated as the following binary quadratic program:
\begin{align}
\label{eq:obj-sahlp}\text{min} \quad  & \sum_{k \in N}f_k x_{k} + \sum_{i\in N}\sum_{k \in N} d_{ik}\,(\chi\, O_i+ \delta\, D_i)\, y_{ik} + \sum_{i,k,j,m \in N} \alpha\,w_{ij}d_{km} y_{ik}y_{jm} \\
\label{eq:oneAlloc}\mbox{s.t.} \quad & \sum_{k \in N}y_{ik}= 1 \quad\quad  i \in N\\[1ex]
\label{eq:onlyToHub}&y_{ik}\leq x_{k} \quad\quad i,k \in N\\[1ex]
\label{eq:xbin}&y_{ik} \in \{0,1\}, \ x_{k} \in \{0,1\} \quad\quad i,k \in N.
\end{align}
The objective is to minimize the total costs of the network which includes the costs of setting up the hubs, the costs of collection and distribution of items between the spoke nodes and the hubs, and the costs of transfer between the hubs. Constraints \eqref{eq:oneAlloc} indicate
that each node $i$ is allocated to precisely one hub (i.e. single allocation) while Constraints~\eqref{eq:onlyToHub} enforce that node $i$ is allocated to a node $k$ only if $k$ is selected as a hub node. The binary conditions are enforced by Constraints \eqref{eq:xbin}.

%

In order to solve SAHLP, many solution methods have been proposed in the literature. The classical approach to obtain an exact solution is to linearize the quadratic objective function. In \cite{skorin1996tight} and \cite{ernst1996efficient} two mixed-integer linear programming (MILP) formulations for the problem have been proposed which are based on a path and a flow representation, respectively.  The path-based formulation in  \cite{skorin1996tight} has $O(|N|^4)$ variables and $O(|N|^3)$ constraints and its linear programming (LP) relaxation was shown to provide tight lower bounds. However, due to the large number of variables and constraints, the path-based formulation can only be solved for instances of relatively small sizes. Alternatively, the flow-based formulation of \cite{ernst1996efficient} uses only $O(|N|^3)$  variables and  $O(|N|^2)$ constraints to linearize the problem. To formulate the flow-based SAHLP model (SAHLP-flow), new variables $z_{ikm}$ are defined as the total amount of flow originating at node $i$ and routed via hubs located at nodes $k$ then $m$, respectively. SAHLP-flow is formulated as
\begin{align}
 \text{min} ~~ & \sum_{k\in N}f_k x_{k}+\sum_{i\in N}\sum_{k \in N} d_{ik}\,(\chi\, O_i+ \delta\, D_i)\, y_{ik} + \sum_{i\in N}\sum_{k\in N}\sum_{m \in N} \alpha \, d_{km}z_{ikm}\cr
\mbox{s.t.} \quad & \eqref{eq:oneAlloc}, \eqref{eq:onlyToHub}, \eqref{eq:xbin}\cr
\label{eq:flow1} & \sum_{m \in N}z_{ikm}-\sum_{m \in N}z_{imk}= O_iy_{ik}-\sum_{j \in N}w_{ij} y_{jk} \quad \forall i,k\\
\label{eq:flow2} & \sum_{m \in N}z_{ikm}\leq O_iy_{ik} \quad \forall i,k\\
\label{eq:flowy}&z_{ikm}\geq 0 \quad \forall i,k,m.
\end{align}
Similar to SAHLP, the objective function minimizes the hub setup costs, the costs of collection and distribution, and the inter-hub transfer costs. Besides Constraints \eqref{eq:oneAlloc}, \eqref{eq:onlyToHub}, \eqref{eq:xbin} which are also used in SAHLP, Constraints~\eqref{eq:flow1} are flow
balance constraints while Constraints~\eqref{eq:flow2} ensure that a flow is possible from spoke $i$ to hub $k$ only if node $i$ is allocated to hub $k$; see \cite{correia2010single}. Finally, Constraints~\eqref{eq:flowy} indicate the non-negativity restriction on the variables $z$. 

The presented flow-based formulation is typically regarded to be the most effective linearized formulation in order to obtain exact solutions for the single-allocation hub location problem. In our computations we use this simple solution method to solve Step \ref{step:if} in Algorithm \ref{alg:columngenerationlinear}. Note that although in the flow-based formulation the uncertain parameters $w_{ij}$ appear in the constraints, we can use this formulation as an oracle in our algorithm while other methods which require linear programming formulations without uncertainty in the constraints can not make use of it.


The SAHLP splits up naturally in first- and second-stage problems as the decision variables in the SAHLP are subject to different planning horizons as discussed above. Therefore, the two-stage robust SAHLP can be modeled as follows:
\begin{equation}\label{eq:hub-two-stageprob}\tag{SAHLP-2RP}
\begin{aligned}
\min_{x \in \{0,1\}^N}\max_{w \in U}\min_{y \in Y(x)} \sum_{k \in N}f_k x_{k} & + \sum_{i\in N}\sum_{k \in N} d_{ik}\,(\chi\, O_i+ \delta\, D_i)\, y_{ik} \\
& + \sum_{i,k,j,m \in N} \alpha\,w_{ij}d_{km} y_{ik}y_{jm},
\end{aligned}
\end{equation}
where 
\[Y(x) =  \{y \in \{0,1\}^{N \times N} : \sum_{k \in N} y_{ik}= 1, y_{ik}\leq x_{k} \quad\forall i,k \in N \}.\]
We assume that $U\subset\R_+^{n^2}$ is a convex uncertainty set. Note that this classical formulation is a quadratic two-stage robust problem. To solve Problem \eqref{eq:hub-two-stageprob} we use the branch \& bound procedure described in Section \ref{sec:lineartwostage}. To this end lower bounds can be calculated by Algorithm \ref{alg:columngenerationlinear} implementing the flow linearization SAHLP-flow in CPLEX (\cite{cplex128}) to solve the oracle in Step \ref{step:if}. The variable fixations in each node of the branch \& bound tree can be added as constraints to the SAHLP-flow formulation. 
\subsubsection{Computational Results}\label{sec:computations}
In this section we apply the branch \& bound method derived in Section \ref{sec:branchandboundlinear} and the CCG method presented in Section \ref{sec:CCG_linear} to the SAHLP. Both algorithms were implemented in C++. For the branch \& bound procedure we calculate the lower and upper bounds by Algorithm \ref{alg:columngenerationlinear} as discussed in the previous sections. The dual solution $\bar x$ is calculated as presented in \eqref{eq:dualsolutionavg}. The branching is performed on the variable $\bar x_i$ which is closest to $0.5$. A feasible solution is calculated by rounding the entries of $\bar x$ to the closest integer value. Note that by this rounding procedure we always obtain a feasible first-stage solution for the SAHLP since we do not have restrictions on the first-stage variables. For the selection of the next branch \& bound node to be processed we use the best-first strategy, i.e. the node with the smallest dual bound is processed next.

For the CCG algorithm we implemented Problem \eqref{eq:lowerboundCCG} in CPLEX 12.8 while Problem \eqref{eq:worstcasescenarioCCG} is solved by Algorithm \ref{alg:columngenerationlinear}. In Algorithm \ref{alg:columngenerationlinear} the dual problem in Step \ref{step:dual} is solved by CPLEX 12.8 \cite{cplex128}. As deterministic oracle in Step \ref{step:if} we use the flow linearization SAHLP-flow presented in Section \ref{sec:hub} which was also implemented in CPLEX 12.8. After termination of Algorithm \ref{alg:columngenerationlinear} we delete all solutions $z$ from the calculated set $Z'$ which have a non-zero slack in the dual problem in Step \ref{step:dual}, i.e. for which $f(z,c^*) > \mu^*$ in the last iteration of Algorithm \ref{alg:columngenerationlinear}. By dualizing the dual problem in Step \ref{step:dual} it can be shown that the optimal value does not change by throwing out all calculated solutions with non-zero slack.
\paragraph{Generation of Random Instances}
We generated random instances as follows: As basis for our instances we use a selection of instances of the AP and the CAB datasets which were intensively studied in the hub location literature. The AP instances are based on the mail flows of Australia Post and were introduced in \cite{ernst1996efficient}. The CAB instances contain airline passenger interactions between $25$ major cities in the United States of America and were first studied in \cite{o1987quadratic}. Both datasets can be found in \cite{orlibrary}. Since there is only one CAB instance available, we introduce three additional instances (cab1 to cab3) by varying the demand values as follows: For each node pair $i,j\in N$, the demand values are drawn randomly from the interval $[0.01 \bar w_{ij}, 10 \bar w_{ij}]$, where $\bar w_{ij}$ is the demand value of the original cab instance. 
The number of locations $n$ together with its pairwise distances $d_{ij}$ are given by the instance data. The set-up costs for hub locations are also given by the instance data in case of the AP instances. Accordingly to \cite{alumur2012hub}, the set-up cost at node $k$ are set to $15 \log(O_k)$ for the CAB instances. The collection, transfer and distribution costs are set to $\chi=3$, $\alpha=0.75$ and $\delta = 2$ for the AP instances while for the CAB instances $\chi=1$, $\delta=1$ and $\alpha$ is varied in $\left\{ 0.2,1\right\}$. For each instance and each $\Gamma\in\left\{ 0.02n^2, 0.1n^2\right\}$, rounded down if fractional, we generate $10$ random budgeted uncertainty sets which are defined by
\[
U_\Gamma = \left\{ w\in \R^{n^2} \ | \ w_{ij} = \bar w_{ij} + \delta_{ij}\hat w_{ij}, \ \sum_{i,j\in N}\delta_{ij}\le \Gamma, \ \delta_{ij}\in [0,1]\right\} .
\]
Here $\bar w$ are the flows given by the AP or CAB instances, respectively, while $\hat w_{ij}$ is chosen randomly in $[0,\bar w_{ij}]$ for each $i,j\in N$, i.e. the change in demand can be at most $100\%$ of the given mean $\bar w_{ij}$.
\paragraph{Analysis of Results}
\begin{table}[htb]
\centering
\resizebox{0.9\textwidth}{!}
{
\begin{tabular}{rcc|rrrrrrrrr}
Inst.& $n$& $\Gamma$& $\Delta_{\text{ad}}$(\%)& $t$(s)& $\#$Nodes& $\Delta_{\text{root}}$(\%)& $\#$Oracle& $i_{\text{lb}}$ & $i_{\text{ub}}$ & $\#$Sol. & $\Delta$(\%) \\
\hline
10LL &10&2&3.4&0.9&2.2&2.5&8.4&2.3&1.6&1.0&0.0 \\
10LL &10&10&10.8&1.8&3.6&4.5&15.2&2.8&1.4&1.0&0.0\\
\hline
20LL&20&8&5.0&3.4&1.0&0.0&3.0&2.0&1.0&1.0&0.0\\
20LL&20&40&14.5&10.4&2.4&9.4&8.8&2.4&1.0&1.0&0.0\\
\hline
25LL &25&12&4.4&10.1&1.0&0.0&4.0&2.0&2.0&1.0&0.0\\
25LL &25&62&13.3&11.7&1.0&0.0&4.4&2.2&2.2&1.2&0.0\\
\hline
40LL &40&32&5.9&150.5&1.2&26.6&3.9&2.1&1.0&1.0&0.0\\
40LL &40&160&15.1&223&1.6&5.6&6.0&2.6&1.1&1.0&0.0\\
\hline
50LL &50&50&7.0&530.3&1.4&10.7&5.6&2.3&1.9&1.0&0.0\\
50LL &50&250&17.1&1308.7&3.2&33.9&13.3&2.4&1.6&1.2&0.0\\
\hline
60LL &60&72&8.3&888.9&1.0&0.0&4.0&2&2.0&1.0&-\\
60LL &60&360&19.1&1001.3&1.0&0.0&4.0&2.0&2.0&1.0&-\\
\hline
70LL&70&98&7.9&1977.4&1.0&0.0&4.2&2.1&2.1&1.1&-\\
70LL&70&490&18.5&8632.2&3.2&11.6&17.0&3.0&2.1&1.0&-\\
\hline
75LL&75&112&8.5&5956.1&1.8&7.6&7.5&2.4&1.9&1.0&-\\
75LL&75&562&18.8&3349.2&1.0&0.0&4.0&2.0&2.0&1.0&-\\
\hline
90LL&90&162&9.8&7460.1&1.0&0.0&4.0&2.0&2.0&1.0&-\\
90LL&90&810&21.1&12681.1&1.6&0.0&6.7&2.1&2.0&1.0&-\\
\end{tabular}
 }
\caption{Results of the branch \& bound procedure for AP instances.}
\label{tbl:AP_instances}
\end{table}
\begin{table}[htb]
\centering
\resizebox{0.9\textwidth}{!}
{
\begin{tabular}{rcc|rrrrrrrrr}
Inst. & $\Gamma$ & $\alpha$ & $\Delta_{\text{ad}}$(\%)&  $t$(s) & $\#$Nodes & $\Delta_{\text{root}}$(\%) & $\#$Oracle & $i_{\text{lb}}$ & $i_{\text{ub}}$ & $\#$Sol. & $\Delta$(\%) \\ \hline 
cab0&12&0.2&4.6&17.0&1.4&1.4&4.7&2.2&1.1&1.0&0.0\\
cab0&12&1.0&9.7&217.9&2.6&9.4&11.8&2.7&1.8&1.0&0.0\\
cab0&62&0.2&12.0&21.1&1.8&3.3&6.3&2.1&1.2&1.0&0.0\\
cab0&62&1.0&22.0&258.3&2.8&4.8&11.8&2.3&1.8&1.0&0.0\\
\hline
cab1&12&0.2&5.1&19.3&1.6&1.7&6.6&2.1&1.9&1.0&0.0\\
cab1&12&1.0&11.3&405.1&7.8&5.2&36.9&2.8&1.9&1.0&0.0\\
cab1&62&0.2&12.8&11.3&1.0&0.0&4.0&2.0&2.0&1.0&0.0\\
cab1&62&1.0&23.4&353.9&5.2&6.2&25.7&2.8&2.1&1.1&0.0\\
\hline
cab2&12&0.2&4.9&19.5&1.4&1.6&4.7&2.2&1.1&1.0&0.0\\
cab2&12&1.0&10.8&145.6&2.0&6.1&8.3&2.3&1.8&1.0&0.0\\
cab2&62&0.2&12.9&17.4&1.4&1.7&4.7&2.2&1.1&1.0&0.0\\
cab2&62&1.0&23.6&114.2&1.4&0.4&6.1&2.4&2.1&1.2&0.0\\
\hline
cab3&12&0.2&6.2&57.5&3.4&29.8&14.9&2.9&1.4&1.0&0.0\\
cab3&12&1.0&10.8&171.8&2.6&7.4&11.7&2.4&2.1&1.1&0.0\\
cab3&62&0.2&14.7&27.5&1.8&5.1&6.3&2.4&1.1&1.0&0.0\\
cab3&62&1.0&23.3&173.3&2.4&2.8&11.8&2.6&2.3&1.3&0.0\\
\end{tabular}
 }
\caption{Results of the branch \& bound procedure for CAB instances. All instances have $n=25$ locations.}
\label{tbl:CAB_instances}
\end{table}

The results for the branch \& bound procedure are presented in Table \ref{tbl:AP_instances} and \ref{tbl:CAB_instances}. Each row shows the average over all $10$ random instances of the following values from left to right: The instance name; the number of locations $n$ for the AP instances; the value $\Gamma$ of the budgeted uncertainty set $U_\Gamma$; the value of $\alpha$ for the CAB instances; the adaptivity gap $\Delta_{\text{ad}}$ in \%, i.e. the percental difference between the optimal value of Problem \eqref{eq:two-stageprob} and the deterministic problem with weights $\bar w$; the total solution time $t$ in seconds; the number of nodes solved in the branch \& bound tree; the percental difference $\Delta_{\text{root}}$ of the lower bound calculated in the root problem to the optimal value of Problem \eqref{eq:two-stageprob}; the total number of oracle calls; the average number of iterations $i_{\text{lb}}$ of Algorithm \ref{alg:columngenerationlinear} to calculate the lower bounds; the average number of iterations $i_{\text{ub}}$ of Algorithm \ref{alg:columngenerationlinear} to calculate the upper bounds; the number of solutions returned by the branch \& bound method or the number of iterations of the CCG, respectively; the average percental difference $\Delta$ (over $10$ random scenarios in $U_\Gamma$) between the best solution in $Z'$ and the deterministic optimal solution in each scenario. To be more precicely, to obtain the value $\Delta$ we generate $10$ random scenarios in $U_\Gamma$ by the following procedure: We first create $n^2$ equally distributed random numbers $s_{i}$ in $[0,\Gamma]$ and define $s_0:=0$. Assume the numbers are given in increasing order. We then define $\delta_i:= s_i-s_{i-1}$. If $\delta\le \1$ is not true we start the procedure again. The random scenario is then given by $w$ with
\[
w_{ij} = \bar w_{ij} + \delta_{in+j} \hat w_{ij} .
\]
After generating $10$ random scenarios $w^1,\ldots w^{10}$, in each scenario we compare the costs of the best solution in $Z'$ to the costs of the optimal solution in the scenario, i.e. for the optimal first-stage solution $\bar x$ we define
\[
\Delta_l:= \frac{\min_{(\bar x,y)\in Z'} f(\bar x, y , w^l) - \min_{y\in Y(\bar x)} f(\bar x,y, w^l)}{\min_{(\bar x,y)\in Z'} f(\bar x, y , w^l)}
\]
and set $\Delta$ to the average of all $\Delta_l$. For the CCG algorithm we define $Z'$ as the set of solutions calculated in the last iteration by Problem \eqref{eq:lowerboundCCG}. Note that since the set of optimal second-stage solutions in $Z'$ is not unique and especially may not be the same for both algorithms, the value of $\Delta$ can be different for the branch \& bound procedure and for the CCG.

The results for the AP instances are shown in Table \ref{tbl:AP_instances}. The adaptivity gap increases with $\Gamma$ and with the dimension. The number of calculated nodes in the branch \& bound tree are in most cases close to $1$ and seems to remain constant with increasing dimension. Nevertheless the run-time increases with the dimension and with $\Gamma$ which is mainly due to the increasing run-time of Algorithm \ref{alg:columngenerationlinear}. Here with higher dimension the calculation time of the deterministic problem increases, while with increasing $\Gamma$ the number of iterations of Algorithm \ref{alg:columngenerationlinear} increases which was already observed in \cite{buchheimkurtzconvex,eufingerrobust}. Another positive observation is that the root gap is very small in general, mostly $0$ and never larger than $34\%$. The number of iterations of Algorithm \ref{alg:columngenerationlinear} is larger for the calculations of the lower bound than for the upper bound, which is because not all hub variables are fixed in the former case. Nevertheless the number of iterations is very low and never larger than $2.2$ for the lower bound and $1.2$ for the upper bound. This leads to a very small number of policies calculated by Algorithm \ref{alg:columngenerationlinear} and to a very small number of oracle calls in total. Finally the values of $\Delta$ indicate that the returned second-stage solutions are optimal in most of the scenarios, as $\Delta$ is $0$ for most of the instances. Note that for larger dimensions due to the time consuming computations we did not determine the $\Delta$ values.

The computations for the CAB instances are presented in Table \ref{tbl:CAB_instances}. The results look similar to the results related to the AP instances. The adaptivity gap is larger for larger values of $\alpha$ and $\Gamma$. The root gap is again very small for most of the instances and never larger than $30\%$. The number of nodes in the branch \& bound tree is very low, but in general higher than for the AP instances. Nevertheless it is never larger than $8\%$ in average. In contrast to the AP instances the total run-time does not increase much with increasing $\Gamma$. Instead the run-time increases significantly with increasing $\alpha$. The reason for this is the larger number of iterations performed by Algorithm \ref{alg:columngenerationlinear} to calculate the lower and the upper bounds. 
Comparing the calculated solutions to the optimal values on random scenarios, the percental difference $\Delta$ is again very close to $0$ for all of the instances.
\begin{table}[htb]
\centering
\resizebox{0.6\textwidth}{!}
{
\begin{tabular}{rcc|rrrrr}
Inst. & $n$ & $\Gamma$ & $t$(s) & $t_{\text{lb}}$(s) & $t_{\text{ub}}$(s) & $\#$Iter. & $\Delta$(\%) \\ \hline 
10LL &10&2&1.8&0.4&0.1&3.7&0.0 \\
10LL &10&10&5.9&1.1&0.1&4.3&0.0 \\
\hline
20LL&20&8&9.1&2.3&0.4&3.0&0.0 \\
20LL&20&40&73.6&14.6&0.4&3.8&0.0 \\
\hline
25LL &25&12&31.1&8.5&1.1&3.0&0.0 \\
25LL &25&62&37.3&10.5&1.1&3.0&0.0 \\
\hline
40LL &40&32&4682.0&1093.9&5.3&3.9&0.0 \\
40LL &40&160&2660.3&720.9&5.5&3.3&0.0 \\
\hline
50LL&50&50&8606.9&2230.8&15.2&3.6&-\\
50LL&50&250&57557.4&12632.5&14.1&4.1&-\\
\end{tabular}
 }
\caption{Results of the CCG algorithm for AP instances.}
\label{tbl:AP_instances_CCG}
\end{table}
\begin{table}[htb]
\centering
\resizebox{0.5\textwidth}{!}
{
\begin{tabular}{rcc|rrrrr}
Inst. & $\Gamma$ & $\alpha$ & $t(s)$ & $t_{\text{lb}}$(s) & $t_{\text{ub}}$(s) & $\#$Iter. & $\Delta$(\%) \\ \hline 
cab&12&0.2&56.6&15.7&0.9&3.1&0.0\\
cab&12&1&4526.9&850.7&1.1&4.0&0.0\\
cab&62&0.2&78.8&19.8&0.9&3.4&0.0\\
cab&62&1&9663.3&1798.0&1.2&4.1&0.0\\
\hline
cab1&12&0.2&53.1&12.8&1.0&3.2&0.0\\
cab1&12&1&3131.2&660.4&1.0&4.6&0.0\\
cab1&62&0.2&32.2&8.8&1.0&3.0&0.0\\
cab1&62&1&10760.1&1914.9&1.0&5.0&0.0\\
\hline
cab2&12&0.2&67.3&18.2&0.8&3.2&0.0\\
cab2&12&1&1616.5&402.0&1.0&3.6&0.0\\
cab2&62&0.2&63.0&16.2&0.8&3.3&0.0\\
cab2&62&1&1000.3&294.4&1.1&3.2&0.0\\
\hline
cab3&12&0.2&285.4&57.3&0.9&4.0&0.0\\
cab3&12&1&1382.1&366.1&1.0&3.5&0.0\\
cab3&62&0.2&121.4&30.9&0.8&3.5&0.0\\
cab3&62&1&1899.3&492.3&1.1&3.7&0.0\\
\end{tabular}
 }
\caption{Results of the CCG algorithm for CAB instances. All instances have $n=25$ locations.}
\label{tbl:CAB_instances_CCG}
\end{table}
\begin{figure}[htb]
\centering
\includegraphics[width=0.6\textwidth]{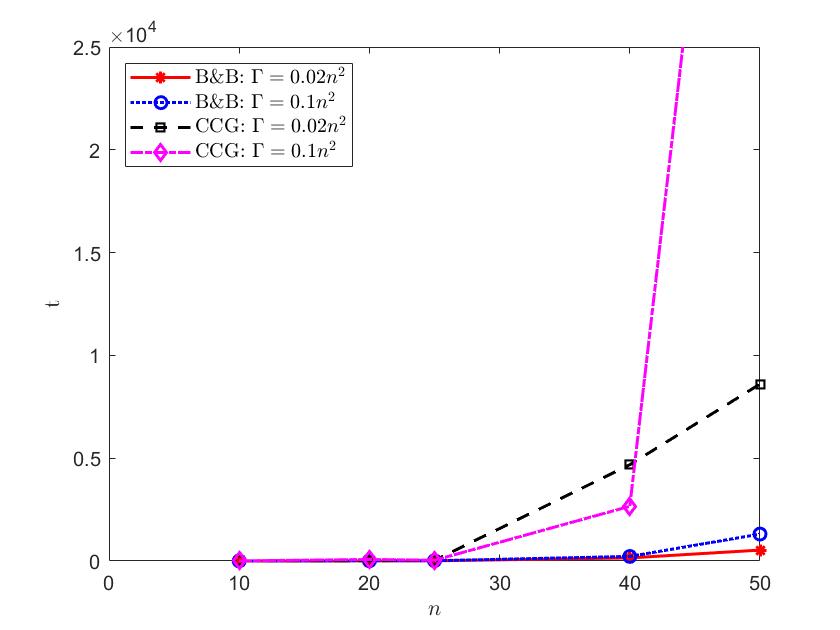}
\caption{Development of the runtime in seconds of both algorithms.}
\label{fig:Runtimes}
\end{figure} 

All results for the CCG algorithm are presented in Table \ref{tbl:AP_instances_CCG} and \ref{tbl:CAB_instances_CCG}. Each row shows the average over all $10$ random instances of the following values from left to right: The instance name; the number of locations $n$ for the AP instances; the value $\Gamma$ of the budgeted uncertainty set $U_\Gamma$; the value of $\alpha$ for the CAB instances; the total solution time $t$ in seconds; the average time $t_{\text{lb}}$ in seconds to solve the lower bound Problem \eqref{eq:lowerboundCCG}; the average time $t_{\text{ub}}$ in seconds to solve the upper bound Problem \eqref{eq:worstcasescenarioCCG}; the number of solutions $l$ calculated by Problem \eqref{eq:lowerboundCCG} which is equal to the number of iterations of the CCG algorithm; the average percental difference $\Delta$ (over $10$ random scenarios $\tilde w\in U_\Gamma$) between the best of the solutions calculated in the last iteration by Problem \eqref{eq:lowerboundCCG} and the deterministic optimal solution in each scenario $w$; see the definition of $\Delta$ above.

The results of the CCG algorithm are less convincing. We could solve AP instances up to $50$ locations in reasonable time, while for the branch \& bound procedure we managed to solve instances with $90$ locations. Furthermore the runtime is at least three times as large as for the branch \& bound method for most of the instances and even larger for growing dimension. The same effect holds for the CAB instances. Here the runtime is much higher for the instances with $\alpha=1$. The large runtime of the CCG is mainly caused by the lower bound problem \eqref{eq:lowerboundCCG}. The calculations of the upper bound, solved by Algorithm \ref{alg:columngenerationlinear}, are less time consuming, at most $6$ seconds in average. The number of calculated solutions, i.e. the number of iterations, is slightly larger than for the branch \& bound procedure but still very small, never larger than $5$. A positive effect is that the performance $\Delta$ of the calculated solutions on random scenarios is very close to $0$ for all instances.

In Figure \ref{fig:Runtimes} we compare the runtimes in seconds of both algorithms. The results show that the runtime of the CCG method increases rapidly for more than $25$ locations and is always much larger than the runtime of the branch \& bound method. For the larger value of $\Gamma$ the run-time of the CCG method explodes if $n$ is larger than $40$.

\paragraph{Analysis of Results for Hard Instances}
For the realistic instances calculated above the number of nodes in the branch \& bound tree, the number of iterations of the CCG as well as the number of iterations of Algorithm \ref{alg:columngenerationlinear} is very low. The same effect occurs for most of the randomly generated instances we tested. To test the boundaries of our algorithm we generated further instances which are generated as the instances above with the only difference that the values $\hat w_{ij}$ are randomly drawn in $[0,10 w_{ij}]$, i.e. the uncertainty sets are much larger. Furthermore for the AP instances we varied $\alpha\in\left\{ 0.75, 1.5\right\}$. The results for the branch \& bound procedure are presented in Table \ref{tbl:AP_instances_bigDeviation}. For the CCG algorithm we could not even solve instances with $25$ locations in reasonable time.
\begin{table}[htb]
\centering
\resizebox{0.9\textwidth}{!}
{
\begin{tabular}{rcc|rrrrrrrrr}
Inst.& $n$& $\alpha$& $\Delta_{\text{ad}}$(\%)& $t$(s)& $\#$Nodes& $\Delta_{\text{root}}$(\%)& $\#$Oracle& $i_{\text{lb}}$ & $i_{\text{ub}}$ & $\#$Sol. & $\Delta$(\%) \\ \hline 
10LL&10&0.75&62.2&14.3&21.4&5.5&165.8&5.7&1.9&1.1&0.0\\
10LL&10&1.5&97.9&8.8&11.4&2.9&90.7&5.7&2.5&2.2&0.0\\
\hline
20LL&20&0.75&107.9&152.1&13.4&6.8&113.2&6.6&2.0&1.3&0.0\\
20LL&20&1.5&145.8&192.0&10.0&5.6&101.5&6.1&3.4&2.6&0.1\\
\hline
25LL&25&0.75&105.8&403.3&17.8&3.9&135.6&5.5&1.9&1.1&0.0\\
25LL&25&1.5&142.0&726.5&15.4&6.0&167.4&7.5&3.3&2.6&0.1\\
\hline
40LL&40&0.75&128.2&4095.7&13.2&4.7&111.6&6.4&2.4&1.2&0.1\\
40LL&40&1.5&162.9&5671.4&6.8&1.5&72.2&6.6&3.8&2.3&0.1\\
\hline
50LL&50&0.75&134.2&2542.0&5.4&2.4&29.9&3.3&2.2&1.1&-\\
50LL&50&1.5&174.5&19425.2&6.2&0.5&74.7&7.3&4.5&3.4&-\\
\hline
cab&25&0.2&100.6&884.8&32.8&3.2&313.9&7.4&1.8&1.4&0.0\\
cab&25&1.0&194.8&13718.3&22.0&1.1&424.3&11.1&7.4&4.8&0.2\\
\hline
cab1&25&0.2&114.5&865.6&34.2&5.0&325.1&7.1&2.3&1.8&0.0\\
cab1&25&1.0&219.3&15129.7&22.6&0.6&508.2&12.6&7.8&4.8&0.3\\
\hline
cab2&25&0.2&116.1&570.9&21.0&3.6&186.6&6.5&1.7&1.3&0.0\\
cab2&25&1.0&233.6&16103.7&18.8&1.0&418.4&11.4&7.3&6.0&0.2\\
\hline
cab3&25&0.2&113.7&312.5&12.8&3.0&98.1&5.9&1.5&1.0&0.0\\
cab3&25&1.0&229.9&10954.0&15.0&0.8&254.3&10.0&6.1&4.9&0.2\\
\end{tabular}
}
\caption{Results of the branch \& bound procedure for instances with large deviations and $\Gamma=0.1n^2$.}
\label{tbl:AP_instances_bigDeviation}
\end{table}

\begin{table}[htb]
\centering
\resizebox{0.7\textwidth}{!}
{
\begin{tabular}{rcc|rrrrrrr}
Inst. & $n$ & $\alpha$ & $t(s)$ & $t_{\text{lb}}$(s) & $t_{\text{ub}}$(s) & $\#$Iter. & $\Delta$(\%) \\ \hline 
10LL &10&0.75&455.9&20.9&0.1&14.3&0.0\\
10LL &10&1.5&124.0&9.9&0.1&9.8&0.0\\
\hline
20LL&20&0.75&59150.5&3471.2&0.6&15.5&0.0\\
20LL&20&1.5&8445.2&880.2&0.8&8.2&0.1\\
\end{tabular}
}
\caption{Results of the CCG for instances with large deviations and $\Gamma=0.1n^2$.}
\label{tbl:AP_instances_bigDeviation_CCG}
\end{table}

The results in Table \ref{tbl:AP_instances_bigDeviation} show that the number of nodes in the branch \& bound tree and the number of iterations of Algorithm \ref{alg:columngenerationlinear} are larger than for the realistic instances above but still never get larger than $33$ and $12$ respectively. Both values are larger for the CAB instances. The number of nodes decreases with increasing dimension and with increasing $\alpha$. The same holds for the root gap which is lower than for the realistic instances for most of the instances. Clearly the adaptivity gap is much larger than for the smaller uncertainty sets. Similar to the results above the number of iterations for the calculations of the lower and the upper bounds and therefore the number of total oracle calls seem to be independent of the dimension. The same holds for the number of calculated second-stage solutions. The performance of these solutions over random scenarios is worse than for the realistic instances above, but still very small and never larger than $0.3\%$. For the CAB instances it is larger for $\alpha=1$. For the CCG algorithm the results are not very convincing. Even for instances with $20$ locations finding an optimal solution took more than $16$ hours in average for $\alpha=0.75$. Interestingly here the instances with smaller $\alpha$ were harder to solve.

\begin{figure}[htb]
\centering
\includegraphics[scale=0.2]{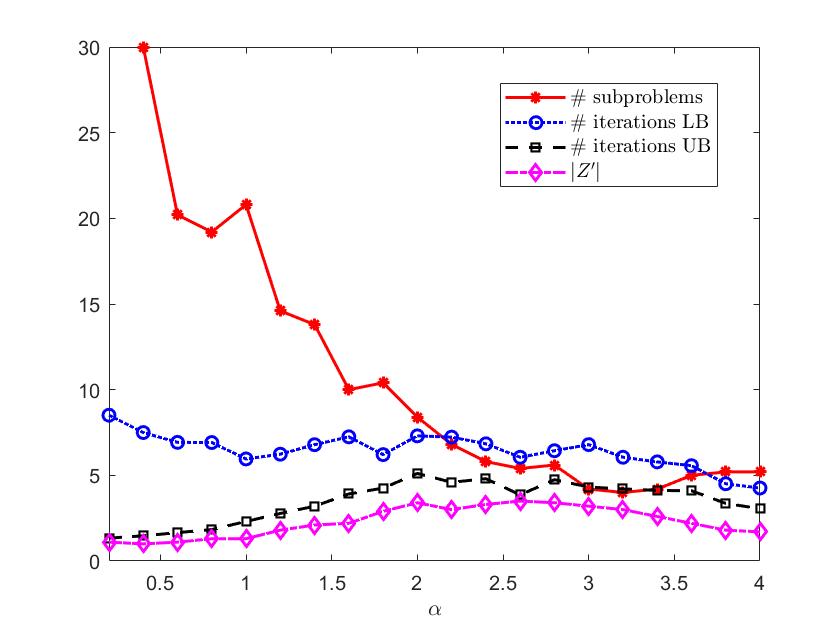}
\includegraphics[scale=0.2]{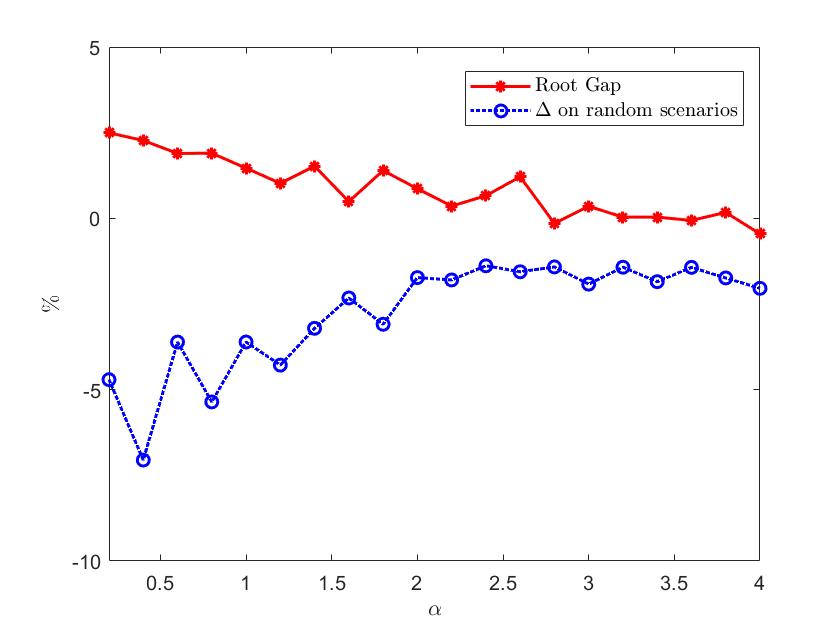}
\caption{Development of the parameters of the branch \& bound procedure over $\alpha$ for the $20LL$ instance with random deviations $\hat w_{ij}\in [0,10 \bar w_{ij}]$, $\chi=3$ and $\delta = 2$. The graphs in the right plot are presented in logarithmic scale.}
\label{fig:VaryAlpha}
\end{figure} 

\begin{figure}[htb]
\centering
\fbox{
\begin{minipage}[b]{.4\linewidth}
\begin{tikzpicture}[scale=1]
\node (0) at (1.29443,1.95227) [draw,circle,fill,inner sep=1pt] {};
\node (1) at (2.34878,1.47492) [draw,circle,fill,inner sep=1pt] {};
\node (2) at (3.10043,2.14583) [draw,circle,fill,inner sep=1pt] {};
\node (3) at (4.97285,1.79463) [draw,circle,fill,inner sep=1pt] {};
\node (4) at (1.82624,3.4231) [draw,circle,fill,inner sep=1pt] {};
\node (5) at (2.72004,2.96833) [draw,circle,fill,inner sep=1pt] {};
\node (6) at (3.26599,3.13228) [draw,rectangle,fill,red,inner sep=2pt] {};
\node (7) at (4.22351,3.28823) [draw,circle,fill,inner sep=1pt] {};
\node (8) at (1.92375,3.95078) [draw,circle,fill,inner sep=1pt] {};
\node (9) at (3.02406,3.81061) [draw,circle,fill,inner sep=1pt] {};
\node (10) at (3.56238,3.88141) [draw,circle,fill,inner sep=1pt] {};
\node (11) at (4.4608,3.83888) [draw,circle,fill,inner sep=1pt] {};
\node (12) at (2.17471,4.47848) [draw,circle,fill,inner sep=1pt] {};
\node (13) at (2.88755,4.74228) [draw,rectangle,fill,red,inner sep=2pt] {};
\node (14) at (3.23611,4.59633) [draw,circle,fill,inner sep=1pt] {};
\node (15) at (4.27557,4.29661) [draw,circle,fill,inner sep=1pt] {};
\node (16) at (1.28903,5.24757) [draw,circle,fill,inner sep=1pt] {};
\node (17) at (2.5061,5.15261) [draw,circle,fill,inner sep=1pt] {};
\node (18) at (3.15925,4.9746) [draw,circle,fill,inner sep=1pt] {};
\node (19) at (3.60101,4.95676) [draw,circle,fill,inner sep=1pt] {};
\draw [-, thick] (0) to  (6);
\draw [-, thick] (1) to  (6);
\draw [-, thick] (2) to  (6);
\draw [-, thick] (3) to  (6);
\draw [-, thick] (4) to  (6);
\draw [-, thick] (5) to  (6);
\draw [-, thick] (7) to  (6);
\draw [-, thick] (8) to  (13);
\draw [-, thick] (9) to  (6);
\draw [-, thick] (10) to  (6);
\draw [-, thick] (11) to  (6);
\draw [-, thick] (12) to  (13);
\draw [-, thick] (14) to  (13);
\draw [-, thick] (15) to  (13);
\draw [-, thick] (16) to  (13);
\draw [-, thick] (17) to  (13);
\draw [-, thick] (18) to  (13);
\draw [-, thick] (19) to  (13);
\end{tikzpicture}
\centering
\caption{Deterministic Solution}
\end{minipage}}
\qquad
\fbox{
\begin{minipage}[b]{.4\linewidth}
\begin{tikzpicture}[scale=1]
\node (0) at (1.29443,1.95227) [draw,circle,fill,inner sep=1pt] {};
\node (1) at (2.34878,1.47492) [draw,rectangle,fill,red,inner sep=2pt] {};
\node (2) at (3.10043,2.14583) [draw,circle,fill,inner sep=1pt] {};
\node (3) at (4.97285,1.79463) [draw,circle,fill,inner sep=1pt] {};
\node (4) at (1.82624,3.4231) [draw,circle,fill,inner sep=1pt] {};
\node (5) at (2.72004,2.96833) [draw,rectangle,fill,red,inner sep=2pt] {};
\node (6) at (3.26599,3.13228) [draw,circle,fill,inner sep=1pt] {};
\node (7) at (4.22351,3.28823) [draw,circle,fill,inner sep=1pt] {};
\node (8) at (1.92375,3.95078) [draw,circle,fill,inner sep=1pt] {};
\node (9) at (3.02406,3.81061) [draw,circle,fill,inner sep=1pt] {};
\node (10) at (3.56238,3.88141) [draw,rectangle,fill,red,inner sep=2pt] {};
\node (11) at (4.4608,3.83888) [draw,circle,fill,inner sep=1pt] {};
\node (12) at (2.17471,4.47848) [draw,circle,fill,inner sep=1pt] {};
\node (13) at (2.88755,4.74228) [draw,rectangle,fill,red,inner sep=2pt] {};
\node (14) at (3.23611,4.59633) [draw,circle,fill,inner sep=1pt] {};
\node (15) at (4.27557,4.29661) [draw,circle,fill,inner sep=1pt] {};
\node (16) at (1.28903,5.24757) [draw,circle,fill,inner sep=1pt] {};
\node (17) at (2.5061,5.15261) [draw,circle,fill,inner sep=1pt] {};
\node (18) at (3.15925,4.9746) [draw,circle,fill,inner sep=1pt] {};
\node (19) at (3.60101,4.95676) [draw,circle,fill,inner sep=1pt] {};
\draw [-, thick, black!100] (0) to  (1);
\draw [-, thick, black!100] (2) to  (5);
\draw [-, thick, black!100] (3) to  (10);
\draw [-, thick, black!100] (4) to  (13);
\draw [-, thick, black!100] (6) to  (10);
\draw [-, thick, black!100] (7) to  (10);
\draw [-, thick, black!100] (8) to  (13);
\draw [-, thick, black!100] (9) to  (10);
\draw [-, thick, black!100] (11) to  (10);
\draw [-, thick, black!100] (12) to  (13);
\draw [-, thick, black!100] (14) to  (13);
\draw [-, thick, black!100] (15) to  (10);
\draw [-, thick, black!100] (16) to  (13);
\draw [-, thick, black!100] (17) to  (13);
\draw [-, thick, black!100] (18) to  (13);
\draw [-, thick, black!100] (19) to  (13);
\end{tikzpicture}
\centering
\caption{Solution $1$}
\end{minipage}
}
\\
\fbox{
\begin{minipage}[b]{.4\linewidth}
\begin{tikzpicture}[scale=1]
\node (0) at (1.29443,1.95227) [draw,circle,fill,inner sep=1pt] {};
\node (1) at (2.34878,1.47492) [draw,rectangle,fill,red,inner sep=2pt] {};
\node (2) at (3.10043,2.14583) [draw,circle,fill,inner sep=1pt] {};
\node (3) at (4.97285,1.79463) [draw,circle,fill,inner sep=1pt] {};
\node (4) at (1.82624,3.4231) [draw,circle,fill,inner sep=1pt] {};
\node (5) at (2.72004,2.96833) [draw,rectangle,fill,red,inner sep=2pt] {};
\node (6) at (3.26599,3.13228) [draw,circle,fill,inner sep=1pt] {};
\node (7) at (4.22351,3.28823) [draw,circle,fill,inner sep=1pt] {};
\node (8) at (1.92375,3.95078) [draw,circle,fill,inner sep=1pt] {};
\node (9) at (3.02406,3.81061) [draw,circle,fill,inner sep=1pt] {};
\node (10) at (3.56238,3.88141) [draw,rectangle,fill,red,inner sep=2pt] {};
\node (11) at (4.4608,3.83888) [draw,circle,fill,inner sep=1pt] {};
\node (12) at (2.17471,4.47848) [draw,circle,fill,inner sep=1pt] {};
\node (13) at (2.88755,4.74228) [draw,rectangle,fill,red,inner sep=2pt] {};
\node (14) at (3.23611,4.59633) [draw,circle,fill,inner sep=1pt] {};
\node (15) at (4.27557,4.29661) [draw,circle,fill,inner sep=1pt] {};
\node (16) at (1.28903,5.24757) [draw,circle,fill,inner sep=1pt] {};
\node (17) at (2.5061,5.15261) [draw,circle,fill,inner sep=1pt] {};
\node (18) at (3.15925,4.9746) [draw,circle,fill,inner sep=1pt] {};
\node (19) at (3.60101,4.95676) [draw,circle,fill,inner sep=1pt] {};
\draw [-, thick, black!100] (0) to  (1);
\draw [-, thick, black!100] (2) to  (5);
\draw [-, thick, black!100] (3) to  (10);
\draw [-, thick, black!100] (4) to  (5);
\draw [-, thick, black!100] (6) to  (10);
\draw [-, thick, black!100] (7) to  (10);
\draw [-, thick, black!100] (8) to  (13);
\draw [-, thick, black!100] (9) to  (10);
\draw [-, thick, black!100] (11) to  (10);
\draw [-, thick, black!100] (12) to  (13);
\draw [-, thick, black!100] (14) to  (13);
\draw [-, thick, black!100] (15) to  (10);
\draw [-, thick, black!100] (16) to  (13);
\draw [-, thick, black!100] (17) to  (13);
\draw [-, thick, black!100] (18) to  (13);
\draw [-, thick, black!100] (19) to  (13);     
\end{tikzpicture}
\centering
\caption{Solution 2}
\end{minipage}
}
\qquad
\fbox{
\begin{minipage}[b]{.4\linewidth}
\begin{tikzpicture}[scale=1]
\node (0) at (1.29443,1.95227) [draw,circle,fill,inner sep=1pt] {};
\node (1) at (2.34878,1.47492) [draw,rectangle,fill,red,inner sep=2pt] {};
\node (2) at (3.10043,2.14583) [draw,circle,fill,inner sep=1pt] {};
\node (3) at (4.97285,1.79463) [draw,circle,fill,inner sep=1pt] {};
\node (4) at (1.82624,3.4231) [draw,circle,fill,inner sep=1pt] {};
\node (5) at (2.72004,2.96833) [draw,rectangle,fill,red,inner sep=2pt] {};
\node (6) at (3.26599,3.13228) [draw,circle,fill,inner sep=1pt] {};
\node (7) at (4.22351,3.28823) [draw,circle,fill,inner sep=1pt] {};
\node (8) at (1.92375,3.95078) [draw,circle,fill,inner sep=1pt] {};
\node (9) at (3.02406,3.81061) [draw,circle,fill,inner sep=1pt] {};
\node (10) at (3.56238,3.88141) [draw,rectangle,fill,red,inner sep=2pt] {};
\node (11) at (4.4608,3.83888) [draw,circle,fill,inner sep=1pt] {};
\node (12) at (2.17471,4.47848) [draw,circle,fill,inner sep=1pt] {};
\node (13) at (2.88755,4.74228) [draw,rectangle,fill,red,inner sep=2pt] {};
\node (14) at (3.23611,4.59633) [draw,circle,fill,inner sep=1pt] {};
\node (15) at (4.27557,4.29661) [draw,circle,fill,inner sep=1pt] {};
\node (16) at (1.28903,5.24757) [draw,circle,fill,inner sep=1pt] {};
\node (17) at (2.5061,5.15261) [draw,circle,fill,inner sep=1pt] {};
\node (18) at (3.15925,4.9746) [draw,circle,fill,inner sep=1pt] {};
\node (19) at (3.60101,4.95676) [draw,circle,fill,inner sep=1pt] {};
\draw [-, thick, black!100] (0) to  (5);
\draw [-, thick, black!100] (2) to  (5);
\draw [-, thick, black!100] (3) to  (10);
\draw [-, thick, black!100] (4) to  (5);
\draw [-, thick, black!100] (6) to  (10);
\draw [-, thick, black!100] (7) to  (10);
\draw [-, thick, black!100] (8) to  (13);
\draw [-, thick, black!100] (9) to  (10);
\draw [-, thick, black!100] (11) to  (10);
\draw [-, thick, black!100] (12) to  (13);
\draw [-, thick, black!100] (14) to  (13);
\draw [-, thick, black!100] (15) to  (10);
\draw [-, thick, black!100] (16) to  (13);
\draw [-, thick, black!100] (17) to  (13);
\draw [-, thick, black!100] (18) to  (13);
\draw [-, thick, black!100] (19) to  (13);
\end{tikzpicture}
\centering
\caption{Solution 3}
\end{minipage}
}
\caption{The optimal solution of the nominal scenario $\bar w$ (top left) and the optimal two-stage robust solution presented by all $3$ solutions in $Z'$ returned by Algorithm \ref{alg:columngenerationlinear} in the optimal branch \& bound node for a $20LL$ instance with $\alpha=1.5$ and $\Gamma=40$.}
\label{fig:allSolutions}
\end{figure}

In Figure \ref{fig:VaryAlpha} we present the development of several problem parameters over $\alpha$ for the $20LL$ instance. All values are the average over $10$ random uncertainty sets with random deviations $\hat w_{ij}\in [0,10 \bar w_{ij}]$. Cost parameters are defined as above by $\chi=3$ and $\delta = 2$. Figure \ref{fig:VaryAlpha} shows that the number of nodes in the branch \& bound tree rapidly decreases with increasing $\alpha$. Furthermore the number of iterations performed by Algorithm \ref{alg:columngenerationlinear} to calculate the upper bounds and the number of returned policies in $Z'$ increases until $\alpha = 2$ and afterwards slowly decreases. 
The number of iterations performed by Algorithm \ref{alg:columngenerationlinear} to calculate the lower bounds is nearly constant and slightly decreases. The root gap of the branch \& bound procedure decreases with increasing $\alpha$ and tends to $0$. In contrast to this the performance of the returned policies in $Z'$, indicated by $\Delta$, seems to get worse with increasing $\alpha$, and seems to be constant for $\alpha \ge 2$. Nevertheless all $\Delta$ values are very small and remain close to $0.2\%$ for $\alpha \ge 2$.

In summary the results show that the number of nodes of the branch \& bound procedure and the number of iterations of Algorithm \ref{alg:columngenerationlinear} is very low for the realistic instances of the SAHLP. Hence we could solve instances with up to $90$ locations in less than $4$ hours. Furthermore the number of calculated policies $|Z'|$ is very low for the hub location problem but they perform very well on random scenarios. For the larger uncertainty sets, the number of nodes of the branch \& bound procedure and the number of iterations of Algorithm \ref{alg:columngenerationlinear} is larger but is still very low compared to the dimension of the problem. Furthermore the latter values seem to be nearly constant with increasing dimension. The runtime and the number of iterations of Algorithm \ref{alg:columngenerationlinear} increase with increasing $\alpha$ while the number of nodes of the branch \& bound tree decreases.

An example of an optimal solution of a random instance with $20$ locations and $\hat w_{ij}$ randomly drawn in $[0,10 \bar w_{ij}]$ is shown in Figure \ref{fig:allSolutions}. The figure shows the optimal solution of the nominal scenario $\bar w$ and the three returned solutions in $Z'$. The number of hubs is larger in the two-stage robust solution than in the deterministic solution since for flexible re-allocation after a scenario occured it can be beneficial to build further hubs in advance. Furthermore the figure indicates that a hub which is used by many locations in the deterministic solution may not be used by the second-stage reactions of the two-stage solution.

\subsection{The Capital Budgeting Problem}\label{sec:CBComputations}
In this section the oracle-based branch \& bound algorithm and the CCG algorithm are exemplarily applied to the two-stage robust capital budgeting problem studied in \cite{arslan2019decomposition} which can be naturally defined as a two-stage problem.

The capital budgeting problem (CB) is an investment planning problem, where a subset of $n$ projects has to be selected. Each project $i\in [n]$ has costs $c_i$ and an uncertain profit $\tilde p_i$ which depends on a set of $m$ risk factors $\xi\in U\subset \R^m$. The profits are given by $\tilde p_i(\xi ) = (1+\frac{1}{2}Q_i^\top\xi)\bar p_i$, where $\bar p_i$ are the nominal profits and $Q_i$ is the $i$-th row of the factor loading matrix. For each project the company can decide if it wants to invest in the project here-and-now or if it wants to wait until the risk factors are known. If an investment is postponed to the second stage the profit generated by the project is $f\tilde p_i$ where $0\le f<1$. The costs of a project are the same in the first and the second stage. The company has a given budget $B$ for investing in projects and can additionally take out a loan of volume $C_1$ with costs $\lambda$ in the first stage and a loan of volume $C_2$ with costs $\lambda\mu$ in the second stage where $\mu > 1$. The aim is to maximize the worst-case profit. This problem can be formulated as
\begin{equation}\label{eq:capitalbudgetingtwostage}
\begin{aligned}
\max_{(x,x_0)\in X} -\lambda x_0 + \bar p^\top(x+fy) + \min_{\xi\in U}\max_{(y,y_0)\in Y((x,x_0))}\sum_{i=1}^{n}\frac{1}{2}Q_i^\top\xi\bar p_i (x_i+fy_i) -\lambda\mu y_0
\end{aligned}
\end{equation}
where $X=\left\{ (x,x_0)\in\binvar{n+1} \ | \ c^\top x \le B + C_1x_0\right\}$ is the set of feasible first-stage solutions and 
\[Y((x,x_0))=\left\{ (y,y_0)\in \binvar{n+1} \ | \ c^\top (x+y) \le B + C_1x_0 + C_2 y_0, \ x+y\le \boldsymbol{1}\right\}\]
is the set of feasible second-stage solutions. For more details see \cite{arslan2019decomposition}.

\subsubsection{Computational Results}\label{sec:computations}
In this section we apply the branch \& bound method derived in Section \ref{sec:lineartwostage} and the CCG method presented in Section \ref{sec:CCG_linear} to the capital budgeting problem. The implementation of both algorithms is the same as in Section \ref{sec:hub}. As deterministic oracle in Step \ref{step:if} of Algorithm \ref{alg:columngenerationlinear} we implemented the deterministic version of the integer programming formulation of Problem \eqref{eq:capitalbudgetingtwostage} in CPLEX 12.8. Note that since we consider a maximization problem here the terms upper bound and lower bound are swapped. 

We compare both variants of calculating a dual solution $\bar x$ presented in \eqref{eq:dualsolutionavg} and \eqref{eq:dualsolutionopt} which we denote by \textit{DualSol-Avg} and \textit{DualSol-Opt}, respectively. The branching is performed on the variable which is closest to $0.5$. A feasible first-stage solution is obtained by rounding the entries of $\bar x$ to the closest integer value. If this solution is not feasible we choose the first solution which was returned by Algorithm \ref{alg:columngenerationlinear} after the calculation of the upper bound.

For our tests we use the original instances studied in \cite{arslan2019decomposition}. The authors generate random instances with $n\in\left\{ 10,20,30,40,50,100\right\}$ projects and $m\in\left\{ 4,6,8\right\}$ risk factors. For each combination $20$ instances are generated. The uncertainty set is given by the box $U=[-1,1]^m$. For more details see \cite{arslan2019decomposition}.

\paragraph{Analysis of Results}
\begin{table}[htb]
\centering
\resizebox{0.9\textwidth}{!}
{
\begin{tabular}{cc|rrrr|rrrr}
& & \multicolumn{4}{|c|}{DualSol-Avg} & \multicolumn{4}{c}{DualSol-Opt}\\
\hline
$n$ & $m$ & $\#$Nodes &  $\#$Oracle & $t$(s) & Opt(\%) & $\#$Nodes &  $\#$Oracle & $t$(s) & Opt(\%) \\ \hline 
10&4&36.6&175.4&1.8&\bf{100}&\bf{31.0}&\bf{131.2}&\bf{1.5}&\bf{100}\\
10&6&28.5&135.8&1.4&\bf{100}&\bf{20.5}&\bf{85.5}&\bf{1.0}&\bf{100}\\
10&8&27.1&140.2&1.6&\bf{100}&\bf{21.3}&\bf{105.3}&\bf{1.3}&\bf{100}\\
\hline
20&4&341.7&3584.3&61.7&\bf{100}&\bf{259.1}&\bf{2591.7}&\bf{50.3}&\bf{100}\\
20&6&130.8&1609.9&30.0&\bf{100}&\bf{110.2}&\bf{1185.0}&\bf{25.6}&\bf{100}\\
20&8&233.3&3000.0&58.0&\bf{100}&\bf{185.3}&\bf{2177.3}&\bf{48.0}&\bf{100}\\
\hline
30&4&1567.1&24812.2&968.5&\bf{95}&\bf{866.9}&\bf{12561.9}&\bf{589.7}&\bf{95}\\
30&6&4883.5&98070.3&2224.7&90&\bf{2613.3}&\bf{49107.6}&\bf{1245.6}&\bf{100}\\
30&8&7551.5&165807.6&3764.3&\bf{90}&\bf{5587.5}&\bf{118041.0}&\bf{2871.5}&\bf{90}\\
\hline
40&4&172.9&2454.8&101.4&\bf{100}&\bf{79.1}&\bf{991.7}&\bf{47.8}&\bf{100}\\
40&6&1099.6&24308.3&867.8&\bf{100}&\bf{415.1}&\bf{9053.3}&\bf{390.5}&\bf{100}\\
40&8&82702.4&2521663.0&78876.8&\bf{70}&\bf{56771.4}&\bf{1696806.5}&\bf{61148.2}&\bf{70}\\
\hline
50&4&229.4&2849.4&99.2&\bf{100}&\bf{100.3}&\bf{1169.7}&\bf{48.7}&\bf{100}\\
50&6&268.8&5357.1&227.6&\bf{100}&\bf{118.0}&\bf{2160.7}&\bf{137.0}&\bf{100}\\
50&8&2237.4&65036.7&3530.1&80&\bf{439.9}&\bf{11883.3}&\bf{717.9}&\bf{100}\\
\hline
100&4&1191.1&12911.5&536.4&\bf{100}&\bf{340.7}&\bf{3410.8}&\bf{147.4}&\bf{100}\\
100&6&4546.1&73554.6&1977.5&90&\bf{543.3}&\bf{9458.4}&\bf{288.6}&\bf{100}\\
100&8&4289.7&93770.2&2899.9&90&\bf{603.3}&\bf{14755.7}&\bf{581.7}&\bf{100}
\end{tabular}
 }
\caption{Results of the branch \& bound procedure for both variants.}
\label{tbl:CB_BnB_Strategies1}
\end{table}

\begin{table}[htb]
\centering
\resizebox{0.95\textwidth}{!}
{
\begin{tabular}{cc|r|rrrrr|rrrrr}
& & &\multicolumn{5}{|c|}{DualSol-Avg} & \multicolumn{5}{c}{DualSol-Opt}\\
\hline
$n$& $m$& $\Delta_{\text{ad}}$(\%) & $\Delta_{\text{root}}$(\%)& $i_{\text{ub}}$ & $i_{\text{lb}}$ & $\#$Sol. & $\Delta$(\%) & $\Delta_{\text{root}}$(\%)& $i_{\text{ub}}$ & $i_{\text{lb}}$ & $\#$Sol. & $\Delta$(\%) \\
\hline
10&4&13.1&6.8&\bf{2.0}&2.3&\bf{1.8}&\bf{0.3}&\bf{5.5}&\bf{2.0}&\bf{1.9}&\bf{1.8}&0.5\\
10&6&16.7&8.0&\bf{1.7}&2.2&\bf{1.7}&0.4&\bf{7.0}&1.8&\bf{1.8}&\bf{1.7}&\bf{0.3}\\
10&8&18.0&\bf{7.1}&\bf{2.4}&2.2&\bf{1.6}&\bf{0.2}&7.2&2.5&\bf{1.9}&\bf{1.6}&0.3\\
\hline
20&4&6.2&5.3&\bf{4.4}&3.7&\bf{2.5}&\bf{0.6}&\bf{4.8}&4.7&\bf{3.1}&\bf{2.5}&\bf{0.6}\\
20&6&6.5&4.4&\bf{4.7}&3.7&\bf{2.5}&\bf{0.3}&\bf{4.2}&4.8&\bf{2.9}&\bf{2.5}&0.4\\
20&8&7.5&5.2&\bf{4.9}&4.4&\bf{2.5}&\bf{0.3}&\bf{3.8}&5.2&\bf{3.5}&\bf{2.5}&\bf{0.3}\\
\hline
30&4&3.2&5.0&\bf{6.0}&5.0&4.0&\bf{0.6}&\bf{3.7}&6.1&\bf{4.5}&\bf{3.8}&\bf{0.6}\\
30&6&3.5&5.1&\bf{7.2}&6.0&3.8&\bf{0.4}&\bf{3.9}&7.5&\bf{5.2}&\bf{3.7}&\bf{0.4}\\
30&8&4.3&4.9&\bf{8.3}&6.9&\bf{3.8}&0.4&\bf{4.0}&8.8&\bf{5.7}&\bf{3.8}&\bf{0.3}\\
\hline
40&4&1.9&5.4&5.8&5.7&\bf{4.1}&0.8&\bf{4.4}&\bf{5.4}&\bf{5.3}&\bf{4.1}&\bf{0.6}\\
40&6&2.3&5.5&8.0&7.3&\bf{5.2}&0.5&\bf{4.0}&\bf{7.5}&\bf{6.4}&\bf{5.2}&\bf{0.3}\\
40&8&3.3&5.0&\bf{11.1}&8.8&\bf{6.2}&\bf{0.4}&\bf{3.3}&11.5&\bf{7.9}&7.0&\bf{0.4}\\
\hline
50&4&2.3&5.1&5.1&5.7&\bf{4.7}&\bf{0.7}&\bf{3.2}&\bf{5.0}&\bf{5.5}&4.8&0.9\\
50&6&2.5&5.8&7.9&7.9&\bf{6.5}&\bf{0.5}&\bf{3.7}&\bf{7.1}&\bf{7.3}&\bf{6.5}&0.6\\
50&8&2.6&5.5&10.1&9.9&\bf{7.8}&\bf{0.4}&\bf{3.5}&\bf{9.1}&\bf{9.2}&8.0&\bf{0.4}\\
\hline
100&4&1.2&5.6&\bf{3.6}&\bf{5.4}&\bf{5.5}&\bf{0.9}&\bf{1.6}&\bf{3.6}&5.8&6.1&\bf{0.9}\\
100&6&1.1&5.8&5.4&\bf{8.1}&7.4&0.8&\bf{3.1}&\bf{5.2}&8.3&\bf{7.1}&\bf{0.7}\\
100&8&1.5&5.0&8.2&\bf{12.2}&10.9&0.8&\bf{3.0}&\bf{7.9}&12.7&\bf{9.4}&\bf{0.7}
\end{tabular}
 }
\caption{Results of the branch \& bound procedure for both variants.}
\label{tbl:CB_BnB_Strategies2}
\end{table}
\begin{table}[htb]
\centering
\resizebox{0.7\textwidth}{!}
{
\begin{tabular}{cc|rrrrrrr}
$n$ & $m$ & Opt(\%) & Gap(\%) & $t$(s) & $t_{\text{ub}}$(s) & $t_{\text{lb}}$(s) & $\#$Iter. & $\Delta$(\%) \\ \hline 
10&4&100&0.0&0.5&0.0&0.0&7.8&0.2\\
10&6&100&0.0&0.4&0.0&0.0&7.4&0.4\\
10&8&100&0.0&0.7&0.0&0.0&8.4&0.3\\
\hline
20&4&90&0.0&733.2&24.7&0.0&12.7&0.6\\
20&6&100&0.0&49.6&1.9&0.0&12.9&0.4\\
20&8&70&0.1&2171.2&69.1&0.0&18.5&0.3\\
\hline
30&4&90&0.0&832.2&42.4&0.0&14.5&0.6\\
30&6&60&0.1&2900.3&105.1&0.0&18.4&0.5\\
30&8&55&0.3&3695.3&137.3&0.0&18.1&0.4\\
\hline
40&4&100&0.0&105.6&6.7&0.0&10.3&0.7\\
40&6&85&0.0&1280.8&59.4&0.0&17.6&0.5\\
40&8&50&0.2&3956.5&131.1&0.1&30.0&0.6\\
\hline
50&4&95&0.3&445.5&52.8&0.0&9.8&1.1\\
50&6&90&0.0&930.1&44.1&0.1&13.5&0.6\\
50&8&75&0.6&2246.1&127.6&0.1&18.7&0.5\\
\hline
100&4&80&1.6&1508.1&239.5&0.1&6.6&1.2\\
100&6&55&3.6&3490.7&424.1&0.1&7.8&1.3\\
100&8&40&3.2&4348.8&452.7&0.1&10.2&1.3\\
\end{tabular}
 }
\caption{Results of the CCG algorithm.}
\label{tbl:CB_CCG}
\end{table}

The results for the branch \& bound procedure are presented in Table \ref{tbl:CB_BnB_Strategies1} and \ref{tbl:CB_BnB_Strategies2}. Each row in Table \ref{tbl:CB_BnB_Strategies1} shows the average over all $20$ instances of the following values from left to right: The number of projects $n$; the number of risk factors $m$; the total number of nodes solved in the branch \& bound tree; the total number of oracle calls; the total solution time $t$ in seconds; the percentage of instances which could be solved to optimality during the timelimit of $7200$ seconds.

Each row in Table \ref{tbl:CB_BnB_Strategies2} shows the average over all $20$ instances of the following values from left to right: The number of projects $n$; the number of risk factors $m$; the adaptivity gap $\Delta_{\text{ad}}$ in \%; the root-gap $\Delta_{\text{root}}$ in \%; the average number of iterations $i_{\text{ub}}$ of Algorithm \ref{alg:columngenerationlinear} to calculate the upper bounds; the average number of iterations $i_{\text{lb}}$ of Algorithm \ref{alg:columngenerationlinear} to calculate the lower bounds; the number of solutions $|Z'|$ Algorithm \ref{alg:columngenerationlinear} returned for the optimal first-stage solution $x$; the average percental difference $\Delta$ (over $10$ random scenarios in $U$) between the best solution in $Z'$ and the deterministic optimal solution in each scenario; see Section \ref{sec:hub} for a precise definition. All values are presented for both variants, \textit{DualSol-Avg} and \textit{DualSol-Opt}. The bold-faced values indicate which of the two variants is better.

The results in Table \eqref{tbl:CB_BnB_Strategies1} indicate that the \textit{DualSol-Opt} variant performs much better on most of the instances. The larger computational effort which is made to calculate the optimal dual solution does not have an impact on the total run-time since the number of processed nodes in the branch \& bound tree is much smaller. For both variants the number of nodes processed in the branch \& bound tree and the number of oracle calls is significantly larger than for the SAHLP; compare to Section \ref{sec:hub}. Both values and therefore the run-time increase with increasing $m$. Interestingly the instances with dimension $n=30$ and $n=40$ seem to be the hardest to solve. The total run-time for the instances with $m=8$ is very large. Nevertheless for most of the configurations all instances could be solved during the timelimit.

In contrast to the latter results, the values in Table \ref{tbl:CB_BnB_Strategies2} are not much larger than for the SAHLP. The root-gap is better for the \textit{DualSol-Opt} variant for most of the instances but is very small for both methods and at most $8\%$. The number of iterations performed to calculate the upper and the lower bounds and the number of calculated solutions are slightly larger than for the SAHLP but still very small. All values seem to be independent of the size of the dimension and the number of risk factors. The gap $\Delta$ is slightly larger than for the SAHLP but still at most $1\%$.

The results for the CCG are presented in Table \ref{tbl:CB_CCG}. Each row shows the average over all $20$ instances of the following values from left to right: The number of projects $n$; the number of risk factors $m$; the percentage of instances which could be solved to optimality during the timelimit of $7200$ seconds, the optimality gap of the CCG after $7200$ seconds; the total solution time $t$ in seconds (exceeding the timelimit is counted as $7200$ seconds); the average solution time $t_{\text{ub}}$ to calculate the upper bounds; the average solution time $t_{\text{lb}}$ to calculate the lower bounds; the total number of iterations; the average percental difference $\Delta$ (over $10$ random scenarios in $U$) between the best solution in $Z'$ and the deterministic optimal solution in each scenario. Here $Z'$ is the set of solutions calculated by Algorithm \ref{alg:columngenerationlinear} in the last iteration. Note that we stopped the calculations for each instance after $7200$ seconds, since for several instances the memory used by CPLEX was too large. Therefore the run-times can not be compared to the run-times of the branch \& bound method.

As for the SAHLP the results of the CCG algorithm are less convincing. The number of instances solved to optimality during the timelimit is much smaller than for the branch \& bound method, sometimes smaller than $55\%$. Nevertheless the optimality gap after the timelimit is very small, at most $3.6\%$. The number of iterations is small for most of the instances and seems to be independent of the size of the dimension. It increases with increasing $m$. As for the SAHLP most of the run-time is used to calculate the upper bound problem. The gap $\Delta$ is smaller than $1\%$ for most of the instances, as it is the case for the branch \& bound method.

To summarize, for the two-stage robust capital budgeting problem the number of nodes processed in the branch \& bound tree and the number of oracle calls is significantly larger than for the SAHLP. Nevertheless since the deterministic problem can be solved much faster the total run-time is not larger for the instances with small $m$. Although most of the instances could be solved during the timelimit by the branch \& bound method, the run-time for instances with $m=8$ can be very large. But still the branch \& bound method solves significantly more instances to optimality than the CCG. Nevertheless the optimality gap of the CCG after the timelimit is very small.

\section{Conclusion.}\label{sec:conclusion}

In this paper we derive a branch \& bound procedure to solve robust binary two-stage problems for a wide class of objective functions. We show that the oracle-based column generation algorithm presented in \cite{buchheimkurtzconvex} can be adapted to calculate lower bounds which can be used in a classical branch \& bound procedure. The whole procedure can be implemented for any algorithm solving the underlying deterministic problem. Furthermore we apply the column-and-constraint generation algorithm studied in \cite{zeng2013twostage} to our problem and show that again the oracle-based algorithm in \cite{buchheimkurtzconvex} can be used to solve one step of the procedure.  We test both algorithms on classical benchmark instances of the single-allocation hub location problem and on random instances of the capital budgeting problem. We show that the number of nodes in the branch \& bound tree, the number of iterations of the CCG algorithm as well as the number of iterations of the column generation algorithm is very low for the SAHLP while the number of branch \& bound nodes increases significantly for the capital budgeting problem. Nevertheless our branch \& bound procedure is much faster than the CCG algorithm and can solve larger instances in reasonable time. Furthermore our computational results indicate that for both algorithms the precalculated second-stage solutions perform very well on random scenarios.

%

\paragraph*{Acknowledgements}
We would like to thank the referees for their valuable comments which significantly improved the paper. Furthermore we thank Ayse Nur Arslan and Boris Detienne for providing us their instances on the capital budgeting problem.

%
%

\bibliographystyle{abbrv}      
\bibliography{./my}   


\end{document}